\documentclass[11pt,a4paper,reqno]{amsart}

\usepackage[utf8]{inputenc}

\usepackage{mathptmx}
\usepackage{amssymb}
\usepackage{amsmath}
\usepackage{amsthm}
\usepackage{latexsym}
\usepackage{amsfonts}
\usepackage{mathrsfs}
\usepackage{titlesec}
\usepackage{color}
\usepackage{graphicx}
\usepackage{stmaryrd}
\usepackage{MnSymbol}
\usepackage[pdfencoding=auto]{hyperref}

\newtheorem{thm}{Theorem}[section]
\newtheorem{cor}[thm]{Corollary}
\newtheorem{lem}[thm]{Lemma}

\theoremstyle{definition}
\newtheorem{defn}[thm]{Definition}
\newtheorem{rem}[thm]{Remark}
\newtheorem*{rem*}{Remark}
\newtheorem{ex}[thm]{Example}

\titleformat{\section}{\normalfont\bfseries\centering}{\thesection.}{.25em}{}
\titleformat{\subsection}{\normalfont\bfseries}{\thesubsection.}{.25em}{}
\titleformat{\subsubsection}{\normalfont\it}{\thesubsubsection.}{.25em}{}
\titlespacing{\section}{0pt}{*5}{*1.5}
\numberwithin{equation}{section}
\allowdisplaybreaks

\renewcommand{\emptyset}{\varnothing}

\newcommand{\braces}[1]{{\rm (}#1{\rm )}}
\newcommand{\rmref}[1]{{\rm\ref{#1}}}

\newcommand{\R}{\ensuremath{\mathbb R}}    
\newcommand{\N}{\ensuremath{\mathbb N}}    
\renewcommand{\C}{\ensuremath{\mathbb C}}    

\newcommand{\product}{[\cdot\,,\cdot]}
\newcommand{\hproduct}{(\cdot\,,\cdot)}

\renewcommand{\>}{\rangle}

\newcommand{\aproduct}{\langle\cdot\,,\cdot\rangle}

\newcommand{\calF}{\mathcal F}         
         
\newcommand{\calH}{\mathcal H}

\newcommand{\calK}{\mathcal K}         
\newcommand{\calL}{\mathcal L}

\newcommand{\la}{\lambda}
\newcommand{\veps}{\varepsilon}
\newcommand{\vphi}{\varphi}
\newcommand{\vrho}{\varrho}

\newcommand{\mat}[4]
{
   \begin{pmatrix}
      #1 & #2\\
      #3 & #4
   \end{pmatrix}
}
\newcommand{\vek}[2]
{
   \begin{pmatrix}
      #1\\
      #2
   \end{pmatrix}
}

\renewcommand{\Im}{\operatorname{Im}}
\renewcommand{\Re}{\operatorname{Re}}

\newcommand{\dom}{\operatorname{dom}}

\newcommand{\ran}{\operatorname{ran}}

\newcommand{\sap}{\sigma_{{ap}}}

\renewcommand{\sp}{\sigma_{+}}
\newcommand{\sm}{\sigma_{-}}

\newcommand{\Slra}{\Leftrightarrow}

\newcommand{\downto}{\downarrow}

\newcommand{\ol}{\overline}

\newcommand{\wh}{\widehat}

\newcommand{\dist}{\operatorname{dist}}
\newcommand{\sgn}{\operatorname{sgn}}

\newcommand{\AC}{\operatorname{AC}}

\DeclareMathOperator*{\slim}{s-lim}

\begin{document}
\title[]{Relatively bounded perturbations of J-non-negative operators}

\author[F. Philipp]{Friedrich Philipp}
\address{{\bf F.~Philipp:} Technische Universit\"at Ilmemau, Institute of Mathematics, Germany}
\email{friedrich.philipp@tu-ilmenau.de}
\urladdr{www.tu-ilmenau.de/obc/team/friedrich-philipp}


\begin{abstract}
We improve known perturbation results for self-adjoint operators in Hilbert spaces and prove spectral enclosures for diagonally dominant $J$-self-adjoint operator matrices. These are used in the proof of the central result, a perturbation theorem for $J$-non-negative operators. The results are applied to singular indefinite Sturm-Liouville operators with $L^p$-potentials. Known bounds on the non-real eigenvalues of such operators are improved.
\end{abstract}

\subjclass[2010]{Primary 47B50, 47A55; Secondary 47E05, 34L15}

\keywords{$J$-self-adjoint operator, $J$-non-negative operator, relatively bounded perturbation, spectral enclosure}

\maketitle
\thispagestyle{empty}

\vspace*{-.5cm}
\section{Introduction}
$J$-self-adjoint operators are special generally non-self-adjoint operators in Hilbert spaces; they appear in various applications in mathematics and mathematical phy\-sics such as the Klein-Gordon equation \cite{g,ggh,j93,lnt06,lnt08} and other types of wave equations \cite{ek,jttv,kl1,kl2}, PT-symmetry in quantum mechanics \cite{bb,ks,lt,lt2,ta} \& \cite[Sec.\ 7]{amt}, and self-adjoint analytic operator functions \cite{llmt,lmm00,lmm06,lmm09}, just to name a few.

They also occur naturally as realizations of indefinite Sturm-Liouville expressions \cite{bp,bpt,cl,cn,kmwz,p,z}. As a motivation, let us consider such an indefinite Sturm-Liouville operator:
\begin{equation}\label{e:SLA1}
A : f\mapsto\sgn\cdot(-f'' + qf)
\end{equation}
with a real-valued potential $q\in L^p(\R)$, $p\ge 1$, defined on the maximal domain
\[
\dom A = \{f\in L^2(\R) : f,f'\in\AC_{\rm loc}(\R),\,-f'' + qf\in L^2(\R)\}.
\]
It is well known that the non-real spectrum of $A$ is bounded and consists of isolated eigenvalues, see, e.g., \cite{bp,bst}, and it is of particular interest to find bounds on these eigenvalues.

The first result in this direction has been established in \cite{bpt} for $q\in L^\infty(\R)$. Later, it was refined and immensely generalized to unbounded potentials and a large class of weights in \cite{bst}. Recently, the bound on the imaginary part from \cite{bst} could be further improved in \cite{ci}. The methods in \cite{bst} and \cite{ci}, however, differ significantly from those used in \cite{bpt}. In fact, while the authors in \cite{bst,ci} work explicitly with the differential expressions and operators, the result in \cite{bpt} follows from an abstract theorem on bounded perturbations of $J$-non-negative operators. The idea is simple: we write $A = A_0 + V$, where
\begin{equation}\label{e:SLA_0}
A_0 : f\mapsto -\sgn\cdot f'',\qquad f\in\dom A_0 = H^2(\R),
\end{equation}
and $V$ is the operator of multiplication with $\sgn\cdot q$. If $J$ further denotes multiplication with $\sgn$, then $J = J^* = J^{-1}$ and $A_0$ is $J$-non-negative, i.e., $JA_0$ is self-adjoint and non-negative. Hence, the abstract theorem can be applied.

Here, we proceed in a similar way as in \cite{bpt}: we first prove an abstract theorem on relatively bounded perturbations of $J$-non-negative operators and apply it to the above situation. In this case, we require $q\in L^p(\R)$ with $p\in [2,\infty]$ since these potentials lead to relatively bounded perturbations of $A_0$. As a consequence, we obtain bounds on the non-real eigenvalues of $A$ (see Theorem \ref{t:SL}); these are of the same flavor as those in \cite{bst}, but surprisingly improve them significantly in the case $q\le 0$ (see Figure \ref{f:comp_bst}).


Let us now touch upon the abstract situation in more detail. For this, let $J\neq I$ be a self-adjoint involution on a Hilbert space $(\calH,\hproduct)$, i.e., $J = J^* = J^{-1}$. 
A linear operator $A$ in $\calH$ is called $J$-self-adjoint if the composition $JA$ is a self-adjoint operator in $\calH$. 
In contrast to self-adjoint operators, the spectrum of a $J$-self-adjoint operator is not necessarily real. The only restriction, in general, is its symmetry with respect to the real axis and it is not hard to construct $J$-self-adjoint operators whose spectrum covers the entire complex plane (see, e.g., Example \ref{ex:ex}) or is empty. It is therefore reasonable to consider special classes of $J$-self-adjoint operators---e.g., (locally) definitizable operators \cite{j03,l82}, or, more specific, $J$-non-negative operators. A $J$-self-adjoint operator $A$ with non-empty resolvent set is called $J$-non-negative if $JA\ge 0$. It is well known \cite{l82} that a $J$-non-negative operator has real spectrum and possesses a spectral function on $\ol\R$ with the possible singularities $0$ and $\infty$. 
%
%

Several types of perturbations of $J$-non-negative operators have already been investigated in numerous works. Here, we only mention \cite{abpt,bj05,j88,jl83} for compact and finite-rank perturbations, \cite{ams,amt,bpt,j75} for bounded perturbations and \cite{j98} for form-bounded perturbations. The main result in \cite{bpt} states that a bounded perturbation $A_0+V$ of a $J$-non-negative operator $A_0$ which satisfies a certain regularity condition can be spectrally decomposed into a bounded operator and a $J$-non-negative operator. In particular, the non-real spectrum of the perturbed operator is contained in a compact set $K$ and bounds on $K$ were given explicitly. We call such operators $J$-non-negative over $\ol\C\setminus K$ (cf.\ Definition \ref{d:nnoK}).

Here, our focus lies on {\em relatively bounded} perturbations of the form $A_0 + V$, where $A_0$ is $J$-non-negative and $V$ is both $J$-symmetric and $A_0$-bounded. Our main result, Theorem \ref{t:main}, generalizes that of \cite{bpt} and shows that whenever the $A_0$-bound of $V$ is sufficiently small, then the operator $A_0+V$ is $J$-non-negative over $\ol\C\setminus K$ with a compact set $K$ which is specified in terms of the relative bounds of $V$.

%
A brief outline of the paper is as follows. First of all, we provide the necessary notions and definitions in Section \ref{s:prelims}. In Section \ref{s:ct} we consider a self-adjoint operator $S$ in a Hilbert space and an $S$-bounded operator $T$ with $S$-bound less than one and provide enclosures for the set
\[
K_S(T) := \{\la\in\vrho(S) : \|T(S-\la)^{-1}\|\ge 1\}.
\]
As a by-product, this result yields an improvement of spectral enclosures from \cite{ct} for the operator $S+T$ (cf.\ Corollary \ref{c:improvement} and Remark \ref{r:improvement}). However, the main reason for considering the set $K_S(T)$ is Theorem \ref{t:new} in Section \ref{s:matrix} which states in particular that the non-real spectrum of $J$-self-adjoint diagonally dominant block operator matrices of the form
\begin{equation}\label{e:mmatrix}
S = \mat{S_+}M{-M^*}{S_-}
\end{equation}
is contained in $K_{S_-}(M)\cap K_{S_+}(M^*)$ if both the $S_-$-bound of $M$ and the $S_+$-bound of $M^*$ are less than one. It is now a key observation that our main object of investigation -- the operator $A_0+V$ -- can be written in the form \ref{e:mmatrix}. A renormalization procedure then allows us to derive the main result, Theorem \ref{t:main}. In Section \ref{s:appl} we apply Theorem \ref{t:main} to singular indefinite Sturm-Liouville operators.

\medskip\noindent
{\bf Notation:} Throughout, $\calH$ and $\calK$ stand for Hilbert spaces. In this paper, whenever we write $T : \calH\to\calK$, we mean that $T$ is a linear operator mapping from $\dom T\subset\calH$ to $\calK$, where $\dom T$ does not necessarily coincide with $\calH$ nor do we assume that it is dense in $\calH$. The space of all bounded operators $T : \calH\to\calK$ with $\dom T = \calH$ is denoted by $L(\calH,\calK)$. As usual, we set $L(\calH) := L(\calH,\calH)$. The spectrum of an operator $T$ in $\calH$ is denoted by $\sigma(T)$, its resolvent set by $\vrho(T) =\C\setminus\sigma(T)$. The set of eigenvalues of $T$ is called the {\em point spectrum} of $T$ and is denoted by $\sigma_p(T)$. The {\em approximate point spectrum} $\sap(T)$ of $T$ is the set of all $\la\in\C$ for which there exists a sequence $(f_n)\subset\dom T$ such that $\|f_n\|=1$ for all $n\in\N$ and $(T-\la)f_n\to 0$ as $n\to\infty$. Throughout, for a set $\Delta\subset\C$ and $r > 0$ we use the notation $B_r(\Delta) := \{z\in\C : \dist(z,\Delta)\le r\}$. For $z_0\in\C$ we also write $B_r(z_0) := B_r(\{z_0\}) = \{z\in\C : |z-z_0|\le r\}$. These sets are intentionally defined to be closed.

\section{Preliminaries}\label{s:prelims}
Let $J\in L(\calH)$ be a self-adjoint involution, i.e., $J = J^* = J^{-1}$. The operator $J$ induces an (in general indefinite) inner product $\product$ on $\calH$:
\[
[f,g] = (Jf,g),\qquad f,g\in\calH.
\]
An operator $A$ in $\calH$ is called {\em $J$-self-adjoint} ({\em $J$-symmetric}) if $JA$ is self-adjoint (symmetric, resp.). Equivalently, $A$ is self-adjoint (resp.\ symmetric) with respect to the inner product $\product$. The spectrum of a $J$-self-adjoint operator is symmetric with respect to $\R$, that is, $\la\in\sigma(A)\;\Slra\;\ol\la\in\sigma(A)$. However, more cannot be said, in general. It is easy to construct examples of $J$-self-adjoint operators whose spectrum is the entire complex plane (see, e.g., Example \ref{ex:ex} below). Therefore, the literature usually focusses on special classes of $J$-self-adjoint operators or on local spectral properties as the spectral points of positive and negative type which we shall explain next.

Let $A$ be a $J$-self-adjoint operator. The subset $\sp(A)$ ($\sm(A)$) of $\sigma(A)$ consists of the points $\la\in\sap(A)$ for which each sequence $(f_n)\subset\dom A$ with $\|f_n\|=1$ for all $n\in\N$ and $(A-\la)f_n\to 0$ as $n\to\infty$ satisfies
\[
\liminf_{n\to\infty}\,[f_n,f_n] > 0\qquad\Big(\limsup_{n\to\infty}\,[f_n,f_n] < 0,\text{ resp.}\Big).
\]
A point $\la\in\sp(A)$ ($\sm(A)$) is called a {\em spectral point of positive \braces{negative} type} of $A$ and a set $\Delta\subset\C$ is said to be of positive (negative) type with respect to $A$ if $\Delta\cap\sigma(A)\subset\sp(A)$ ($\Delta\cap\sigma(A)\subset\sm(A)$, resp.). The notion of the spectral points of positive and negative type was introduced in \cite{lmm} (see also \cite{lamm}). It is immediate that $\la\in\sp(A)$ implies that for each $f\in\ker(A-\la)$, $f\neq 0$, we have $[f,f] > 0$. Hence, $\ker(A-\la)$ is a $J$-positive subspace. In fact, much more holds (see \cite{lmm}): We have $\sp(A)\subset\R$ and, if $\la\in\sp(A)$, then
\begin{enumerate}
	\item[(1)] there exists an open (complex) neighborhood $U$ of $\la$ such that each point in $U$ is either contained in $\vrho(A)$ or in $\sp(A)$. In particular, $U\setminus\R\subset\vrho(A)$;\\[-.3cm]
	\item[(2)] there exists a local spectral function\footnote{For a definition of this notion we refer to, e.g., \cite[Definition 2.2]{p2}.} $E$ on $U\cap\R$ such that for each Borel set $\Delta$ with $\ol\Delta\subset U\cap\R$ the projection $E(\Delta)$ is $J$-self-adjoint and $(E(\Delta)\calH,\product)$ is a Hilbert space.
\end{enumerate}
Roughly speaking, the part of the operator $A$ with spectrum in $U\cap\R$ is a self-adjoint operator in a Hilbert space. Similar statements hold for the spectral points of negative type.

\begin{defn}
A $J$-self-adjoint operator $A$ in $\calH$ is said to be {\em $J$-non-negative} if $\vrho(A)\neq\emptyset$ and $\sigma(JA)\subset [0,\infty)$, that is, $(JAf,f)\ge 0$ for $f\in\dom A$ (equivalently, $[Af,f]\ge 0$ for $f\in\dom A$). The operator $A$ is said to be {\em uniformly $J$-positive} if it is $J$-non-negative and $0\in\rho(A)$, i.e., $\sigma(JA)\subset (0,\infty)$.
\end{defn}

It is well known that the spectrum of a $J$-non-negative operator is real and that $(0,\infty)\cap\sigma(A)\subset\sp(A)$ as well as $(-\infty,0)\cap\sigma(A)\subset\sm(A)$, see, e.g., \cite{l82}. Consequently, $A$ possesses a spectral function $E$ on $\ol\R$ with the possible singularities $0$ and $\infty$. The projection $E(\Delta)$ is then defined for all Borel sets $\Delta\subset\ol\R$ for which $0\notin\partial\Delta$ and $\infty\notin\partial\Delta$. The points $0$ and $\infty$ are called the {\em critical points} of $A$. If both $\|E([\veps,1])\|$ and $\|E([-1,-\veps])\|$ are bounded as $\veps\downto 0$, the point $0$ is said to be a {\em regular} critical point of $A$. In this case, the spectral projection $E(\Delta)$ also exists if $0\in\partial\Delta$. A similar statement holds for the point $\infty$. In the sequel, we agree on calling a $J$-non-negative operator {\em regular} if its critical points $0$ and $\infty$ both are regular.

As was shown in \cite{bpt}, the perturbation of a regular $J$-non-negative operator $A_0$ with a bounded $J$-self-adjoint operator $V$ leads to a $J$-self-adjoint operator $A = A_0 + V$ whose non-real spectrum is bounded and for which there exist $r_\pm > 0$ such that $(r_+,\infty)$ is of positive type and $(-\infty,-r_-)$ is of negative type with respect to $A$. Hence, the perturbed operator exhibits the same good spectral properties as a $J$-non-negative operator in the exterior of a compact set. We call such an operator $J$-non-negative in a neighborhood of $\infty$. The following definition makes this more precise. Here, for a set $\Delta\subset\C$ we define $\Delta^* := \{\ol\la : \la\in\Delta\}$. By $\C^+$ we denote the open upper complex half-plane. We also set $\R^+ := (0,\infty)$ and $\R^- := (-\infty,0)$.

\begin{defn}\label{d:nnoK}
Let $K = K^*\subset\C$ be a compact set, $0\in K$, such that $\C^+\setminus K$ is simply connected. A $J$-self-adjoint operator $A$ in $\calH$ is said to be {\em $J$-non-negative over $\ol\C\setminus K$} if the following conditions are satisfied:
\begin{enumerate}
\item[\rm (i)]   $\sigma(A)\setminus\R\subset K$.
\item[\rm (ii)]  $(\sigma(A)\cap\R^\pm)\setminus K\subset\sigma_\pm(A)$.
\item[\rm (iii)] There exist $M > 0$ and a compact set $K'\supset K$ such that for $\la\in\C\setminus(\R\cup K')$ we have
\begin{equation}\label{e:growth}
\|(A-\la)^{-1}\|\,\le\,M\,\frac{(1+|\la|)^2}{|\Im\la|^2}.
\end{equation}
\end{enumerate}
\end{defn}

The relation \eqref{e:growth} means that the growth of the resolvent of $A$ at $\infty$ is of order at most $2$. The order is $1$ if the fraction in \eqref{e:growth} can be replaced by $|\Im\la|^{-1}$.

\begin{rem}\label{r:J-indep}
Note that the notion of $J$-non-negativity over $\ol\C\setminus K$ does not depend on $J$ explicitly, but only on the inner product $\product$. That is, if $\hproduct_0$ is an equivalent Hilbert space scalar product on $\calH$ such that $\product = (J_0\cdot,\cdot)_0$ for some $\hproduct_0$-self-adjoint involution $J_0$, then an operator $A$ in $\calH$ is $J$-non-negative over $\ol\C\setminus K$ in $(\calH,\hproduct)$ if and only if it is $J_0$-non-negative over $\ol\C\setminus K$ in $(\calH,\hproduct_0)$.
\end{rem}

Due to (ii) a $J$-self-adjoint operator $A$ that is $J$-non-negative over $\ol\C\setminus K$ possesses a (local) spectral function $E$ on $\ol\R\setminus K$ with a possible singularity at $\infty$. We say that $A$ is {\em regular at $\infty$} if $\infty$ is not a singularity of $E$. By \cite[Thm.\ 2.6 and Prop.\ 2.3]{bpt} this is the case if and only if there exists a uniformly $J$-positive operator $W$ in $\calH$ such that $W\dom A\subset\dom A$.

In this paper we investigate relatively bounded perturbations of regular $J$-non-negative operators. Recall that an operator $T : \calH\to\calK_1$ is called {\em relatively bounded} with respect to an operator $S : \calH\to\calK_2$ (or simply {\em $S$-bounded}) if $\dom S\subset\dom T$ and
\begin{equation}\label{e:rb}
\|Tf\|^2\,\le\,a\|f\|^2 + b\|Sf\|^2,\qquad f\in\dom S,
\end{equation}
where $a,b\ge 0$. The infimum of all possible $b$ in \eqref{e:rb} is called the {\em $S$-bound} of $T$. It is often convenient to assume that the $S$-bound of $T$ is less than one. Then, if $S$ is closed, also $S+T$ is closed (see \cite[Thm.\ IV.1.1]{k}) and if $S$ is self-adjoint and $T$ symmetric, the sum $S+T$ is self-adjoint. The latter statement is known as the Kato-Rellich theorem (see, e.g., \cite[Thm.\ V.4.3]{k}).

\section{Perturbations of self-adjoint operators}\label{s:ct}
In this section, $S$ always denotes a self-adjoint operator in a Hilbert space $\calH$ and $T : \calH\to\calK$ is an $S$-bounded operator, where $\calK$ is another Hilbert space. We note that in this situation we have
\begin{equation}\label{e:adjoint}
[T(S-\la)^{-1}]^* = \ol{(S-\ol\la)^{-1}T^*}\,\in\,L(\calK,\calH),
\end{equation}
whenever $\la\in\vrho(S)$. The following set will play a crucial role in our spectral estimates in the subsequent sections:
\[
K_S(T) := \left\{\la\in\vrho(S) : \|T(S-\la)^{-1}\|\ge 1\right\}.
\]
Although its proof is elementary it seems that the following result is new.

\begin{lem}\label{l:norm}
Let $\calH$ and $\calK$ be Hilbert spaces, let $S$ be a self-adjoint operator in $\calH$, and let $T : \calH\to\calK$ be such that
\[
\|Tf\|^2\,\le\,a\|f\|^2 + b\|Sf\|^2,\qquad f\in\dom S\subset\dom T,
\]
where $a,b\ge 0$, $b<1$. Then 
\begin{equation}\label{e:besserer}
K_S(T)\,\subset\,B_S(T) := \bigcup_{t\in\sigma(S)}B_{\sqrt{a+bt^2}}(t).
\end{equation}
For $\la\notin B_S(T)$ we have
\begin{equation}\label{e:estimate}
\|T(S-\la)^{-1}\|\,\le\,\sup_{t\in\sigma(S)}\frac{\sqrt{a+bt^2}}{|\la-t|}\,<\,1.
\end{equation}
\end{lem}
\begin{proof}
Let $f\in\dom S^2$. Then
\[
\|Tf\|^2\,\le\,(af,f) + (bS^2f,f) = \|(a+bS^2)^{1/2}f\|^2.
\]
If $\la\in\vrho(S)$, this implies
\begin{align*}
\|T(S-\la)^{-1}\|
&\le\,\|(a+bS^2)^{1/2}(S-\la)^{-1}\| = \sup_{t\in\sigma(S)}\phi_{\la}(t),
\end{align*}
where $\phi_\la(t) := \tfrac{\sqrt{a+bt^2}}{|\la-t|}$. Therefore and due to $\lim_{t\to\pm\infty}\phi_{\la}(t) = \sqrt b < 1$, for $\|T(S-\la)^{-1}\| < 1$ it is sufficient that $\phi_{\la}(t) < 1$ for each $t\in\sigma(S)$. But the latter is just equivalent to $\la$ not being contained in the right-hand side of \eqref{e:besserer}.
\end{proof}

\begin{cor}
Under the assumptions of Lemma \rmref{l:norm} we have
\begin{equation}\label{e:hull}
K_S(T)\,\subset\,\left\{\la\in\C : (\Im\la)^2\le a+\frac{b}{1-b}\,(\Re\la)^2\right\}.
\end{equation}
\end{cor}
\begin{proof}
We prove that $B_{\sqrt{a+bt^2}}(t)$ is contained in the right-hand side of \eqref{e:hull} for arbitrary $t\in\R$. For this, let $\la = \alpha + i\beta\in B_{\sqrt{a+bt^2}}(t)$, i.e., $(\alpha - t)^2 + \beta^2 < a+bt^2$, which implies
\[
\beta^2 < a + 2\alpha t - \alpha^2 - (1-b)t^2\,\le\,a+\frac b{1-b}\alpha^2.
\]
Indeed, the last inequality follows from $(\frac{\alpha}{\sqrt{1-b}} - t\sqrt{1-b})^2\ge 0$.
\end{proof}

\begin{rem}\label{r:fla}
Let $S$ and $T$ be as in Lemma \ref{l:norm} and $\la\in\C\setminus\R$. Let us consider the global behaviour of the function 
\begin{equation}\label{e:Delta}
\vphi_\la(t) := \frac{a+bt^2}{|t-\la|^2}.
\end{equation}
First of all, we always have $\vphi_\la(\pm\infty) = b$. If $b=0$, then $\vphi_\la$ has no local minima but a global maximum at $t = \Re\la$. Let $b > 0$, $\Re\la\neq 0$, and set
\begin{equation}\label{e:mla}
m_\la := \frac{b|\la|^2-a}{2b\cdot\Re\la}.
\end{equation}
Then we have $\vphi_\la(m_\la) = \vphi_\la(\pm\infty) = b$, and $\vphi_\la$ has a global maximum and a global minimum at
\[
t_{\max} = m_\la + \sgn(\Re\la)\sqrt{\frac ab + m_\la^2}
\qquad\text{and}\qquad
t_{\min} = m_\la - \sgn(\Re\la)\sqrt{\frac ab + m_\la^2},
\]
respectively. Let $b > 0$, $\Re\la = 0$. If $(\Im\la)^2 > \tfrac ab$, then $\vphi_\la$ has no local maxima but a global minimum at zero. If $(\Im\la)^2 < \tfrac ab$, then $\vphi_\la$ has no local minima but a global maximum at zero. In case $(\Im\la)^2 = \tfrac ab$, we have $\vphi_\la\equiv b$.
\end{rem}

\begin{cor}\label{c:smallerb}
Let $\calH$ and $\calK$ be Hilbert spaces, let $S$ be a self-adjoint operator in $\calH$ which is bounded from above by $\gamma\in\R$, and let $T : \calH\to\calK$ be such that
\[
\|Tf\|^2\,\le\,a\|f\|^2 + b\|Sf\|^2,\qquad f\in\dom S\subset\dom T,
\]
where $a\ge 0$, $b\in (0,1)$. Then
\begin{equation}\label{e:ba}
K_S(T)\,\subset\,\bigcup_{t\le\gamma}\,B_{\sqrt{a+bt^2}}(t).
\end{equation}
If $|\la| > \gamma + \sqrt{\gamma^2+\tfrac ab}$ and $\Re\la\ge 0$, then
\[
\|T(S-\la)^{-1}\|^2\,\le\,b.
\]
\end{cor}
\begin{proof}
We only have to prove the last claim, which follows from \eqref{e:estimate} if the function $\vphi_\la$ from \eqref{e:Delta} does not exceed the value $b$ on $(-\infty,\gamma]$. Hence, let $|\la| > \gamma + \sqrt{\gamma^2+\tfrac ab}$ and $\Re\la\ge 0$. Then also $|\la|^2 > 2\gamma|\la| + \tfrac ab\ge 2\gamma(\Re\la) + \tfrac ab$. Assume that $\Re\la = 0$. Then $(\Im\la)^2 > \tfrac ab$, and Remark \ref{r:fla} implies that $\vphi_\la(t)\le b$ for all $t\in\R$. Let $\Re\la > 0$, $\Im\la\neq 0$. Then $|\la|^2 > 2\gamma(\Re\la) + \tfrac ab$ is equivalent to $m_\la > \gamma$ with $m_\la$ as defined in \eqref{e:mla}. Note that $\vphi_\la(m_\la) = b$ and that $t_{\max} > m_\la > \gamma$. Hence $\vphi_\la(t)\le b$ for all $t\le\gamma$. For $\Re\la > 0$ and $\Im\la = 0$ the claim follows by continuity of $\mu\mapsto T(S-\mu)^{-1}$ on $\vrho(S)$, which is an immediate consequence of $T(S-\la)^{-1} - T(S-\mu)^{-1} = (\la-\mu)\big[T(S-\la)^{-1}\big](S-\mu)^{-1}$.
\end{proof}

A similar result holds for the case where $S$ is bounded from below. It follows from Corollary \ref{c:smallerb} by considering $-S$.

\begin{figure}[ht]
\begin{center}
\includegraphics[scale=.23]{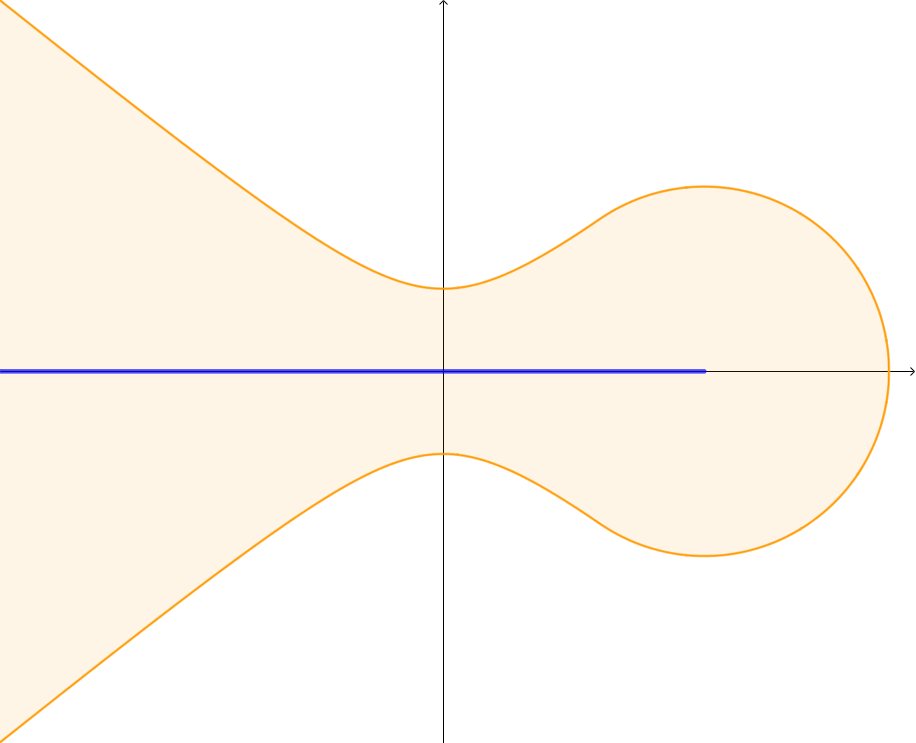}
\end{center}
\caption{Indication of the unbounded enclosure for $K_S(T)$ in \eqref{e:ba} with $a = \gamma = 10$ and $b = 0.4$. The blue segment represents the half-line $(-\infty,\gamma]$.}
\end{figure}

Let us briefly consider the situation where $\calK = \calH$ and thus $T : \calH\to\calH$. If $\la\in\vrho(S)$ and $\|T(S-\la)^{-1}\| < 1$, then
\[
S + T - \la = \big[I + T(S-\la)^{-1}\big]\,(S-\la)
\]
implies that also $\la\in\vrho(S+T)$. By contraposition, $\sigma(S+T)\subset \sigma(S)\cup K_S(T)$. This immediately leads to the following corollary.

\begin{cor}\label{c:improvement}
Let $\calH$ be a Hilbert space, let $S$ be a self-adjoint operator in $\calH$, and let $T : \calH\to\calH$ be such that
\[
\|Tf\|^2\,\le\,a\|f\|^2 + b\|Sf\|^2,\qquad f\in\dom S\subset\dom T,
\]
where $a,b\ge 0$, $b<1$. Then
\begin{equation}\label{e:besser}
\sigma(S+T)\,\subset\,\bigcup_{t\in\sigma(S)}B_{\sqrt{a+bt^2}}(t)\,\subset\,\left\{\la\in\C : (\Im\la)^2\le a+\frac{b}{1-b}\,(\Re\la)^2\right\}.
\end{equation}
\end{cor}

\begin{rem}\label{r:improvement}
Corollary \ref{c:improvement} improves and refines the first two parts of Theorem 2.1 in \cite{ct}. The general spectral inclusion in \cite{ct} is
\[
\sigma(S+T)\,\subset\,\left\{\la\in\C : (\Im\la)^2\,\le\,\frac{a + b(\Re\la)^2}{1-b}\right\}.
\]
In the case where $a>0$ the spectral enclosure \eqref{e:besser} is obviously strictly sharper. However, both boundary curves have the same asymptotes so that the improvement of the spectral inclusion only takes effect for $|\Re\la|$ `not too large' (see Figure \ref{f:comparison} below). In the second part of \cite[Thm. 2.1]{ct} it was proved that if $(\sigma_-,\sigma_+)$ is a spectral gap of $S$ such that $\sigma_-' := \sigma_- + \sqrt{a+b\sigma_-^2} < \sigma_+ - \sqrt{a+b\sigma_+^2} =: \sigma_+'$, then $\{\la : \sigma_-' < \Re\la < \sigma_+'\}\subset\vrho(S+T)$. This now follows immediately from the  first enclosure in \eqref{e:besser}.
\end{rem}

\begin{figure}[ht]
\begin{center}
\includegraphics[scale=.4]{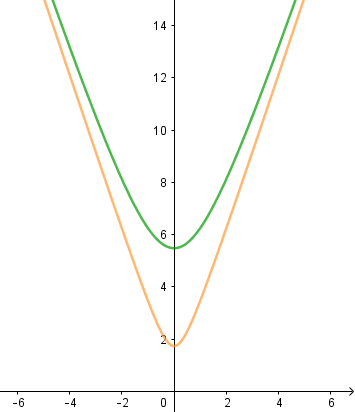}

\end{center}
\caption{The two boundary curves in $\C^+$ from \cite{ct} (green) and \eqref{e:besser} (orange) for $a=3$ and $b=0.9$.}\label{f:comparison}
\end{figure}

\begin{ex}
We consider the massless Dirac operator $H_0 = -i(\sigma_1\partial_x + \sigma_2\partial_y)$ on $\R^2$ with Hermitian matrices $\sigma_1,\sigma_2\in\C^{2\times 2}$ as in \cite[Section 5.1]{ct}. Its domain is given by $\dom H_0 = H^1(\R^2)^2\subset L^2(\R^2)^2$. As in \cite{ct} one shows that for a potential $V\in L^p(\R^2,\C^{2\times 2})$, $p>2$, one has
\[
\|Vf\|_2^2\,\le\,C_p^{\frac{2p}{p-2}}b^{-\frac 2{p-2}}\|f\|_2^2 + b\|H_0f\|_2^2,\qquad f\in\dom H_0,
\]
for every $b\in (0,1)$, where $C_p = \frac{\|V\|_p}{[2\pi(p-2)]^{1/p}}$. By Corollary \ref{c:improvement}, we have
\[
\sigma(H_0+V)\,\subset\,\bigcap_{b\in (0,1)}\left\{\la\in\C : (\Im\la)^2\le C_p^{\frac{2p}{p-2}}b^{-\frac 2{p-2}}+\frac{b}{1-b}\,(\Re\la)^2\right\}.
\]
The envelope of this set in the first quadrant is the curve $\gamma(b) = (x(b),y(b))$, $b\in (0,1)$, where
\[
x(b) = \left(\frac{C_p}{b}\right)^{\frac{p}{p-2}}\cdot\sqrt{\frac{2}{p-2}}(1-b^2)
\qquad\text{and}\qquad
y(b) = \left(\frac{C_p}{b}\right)^{\frac{2}{p-2}}\cdot C_p\cdot\sqrt{\frac{p-2b^2}{p-2}}.
\]
The spectrum of $H_0 + V$ is thus bounded by this curve. Figure \ref{f:ex_comparison} shows the curve $\gamma$ and the spectral inclusion from \cite{ct}.

\begin{figure}[ht]
\begin{center}
\includegraphics[scale=.4]{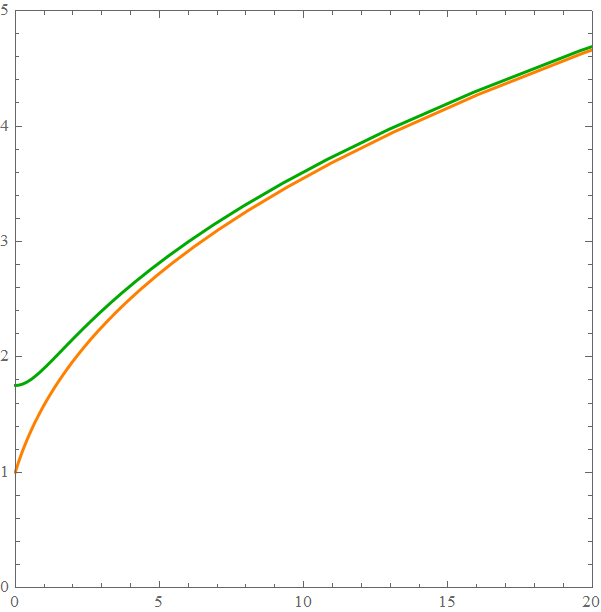}

\end{center}
\caption{The bounding curve $\gamma$ (orange) and the spectral inclusion from \cite{ct} (green) in the case $p=5$ and $\|V\|_5 = 1.8$.}\label{f:ex_comparison}
\end{figure}
\end{ex}

\section{Spectral properties of diagonally dominant J-self-adjoint operator matrices}\label{s:matrix}
In this section we consider a Hilbert space $\calH = \calH_+\oplus\calH_-$ (where $\oplus$ denotes the orthogonal sum of subspaces) and operator matrices of the form
\begin{equation}\label{e:S}
S = \mat{S_+}M{-M^*}{S_-},\qquad\dom S = \dom S_+\oplus\dom S_-,
\end{equation}
where $S_+$ and $S_-$ are self-adjoint operators in the Hilbert spaces $\calH_+$ and $\calH_-$, respectively, and $M : \calH_-\to\calH_+$ is $S_-$-bounded such that also $M^*$ is $S_+$-bounded. In particular, $\dom S_-\subset\dom M$ and $\dom S_+\subset\dom M^*$. Such operator matrices are called {\em diagonally dominant} (cf.\ \cite[Def.\ 2.2.1]{t}). It is clear that the operator $S$ is $J$-symmetric, where $J$ is the involution operator
\[
J = \mat I00{-I},
\]
which induces the indefinite inner product
\[
\left[\vek{f_+}{f_-},\vek{g_+}{g_-}\right] = \left(\mat I00{-I}\vek{f_+}{f_-},\vek{g_+}{g_-}\right) = (f_+,g_+) - (f_-,g_-), \qquad f_\pm,g_\pm\in\calH_\pm.
\]
The following theorem is essentially a generalization of \cite[Thm.~4.3]{glmpt}, where the operator $M$ was assumed to be bounded.

\begin{thm}\label{t:new}
Assume that both the $S_-$-bound of $M$ and the $S_+$-bound of $M^*$ are less than one. Then $S$ is $J$-self-adjoint,
\begin{equation}\label{e:incl}
\sigma(S)\setminus\R\,\subset\,K_{S_-}(M)\,\cap\,K_{S_+}(M^*),
\end{equation}
and
\begin{enumerate}
\item[\rm (i)]  $(\R\cap\vrho(S_-))\setminus K_{S_-}(M)$ is of positive type with respect to $S$.
\item[\rm (ii)] $(\R\cap\vrho(S_+))\setminus K_{S_+}(M^*)$ is of negative type with respect to $S$.
\end{enumerate}
Moreover, if $\la\in\C\setminus\R$, setting
\[
L_+(\la) := M^*(S_+-\la)^{-1}\qquad\text{and}\qquad L_-(\la) := M(S_--\la)^{-1},
\]
we have
\begin{equation}\label{e:resolvent}
\|(S-\la)^{-1}\|\,\le\,\frac{1}{|\Im\la|}\cdot
\begin{cases}
\frac{1 + \|L_-(\la)\| + \|L_-(\la)\|^2}{1-\|L_-(\la)\|^2} &\text{if $\|L_-(\la)\| < 1$}\\
\frac{1 + \|L_+(\la)\| + \|L_+(\la)\|^2}{1-\|L_+(\la)\|^2} &\text{if $\|L_+(\la)\| < 1$}.
\end{cases}
\end{equation}
\end{thm}
\begin{proof}
By assumption, there exist constants $a_\pm,b_\pm\ge 0$, $b_\pm < 1$, such that
\[
\|Mf_-\|^2\,\le\,a_-\|f_-\|^2 + b_-\|S_-f_-\|^2,\qquad f_-\in\dom S_-\subset\dom M
\]
and
\[
\|M^*f_+\|^2\,\le\,a_+\|f_+\|^2 + b_+\|S_+f_+\|^2,\qquad f_+\in\dom S_+\subset\dom M^*.
\]
Define the operators
\[
S_0 := \mat{S_+}00{S_-}\quad\text{and}\quad T := \mat 0M{-M^*}0
\]
with $\dom S_0 = \dom S_+\oplus\dom S_-$ and $\dom T = \dom M^*\oplus\dom M$. Obviously, $S_0$ is $J$-self-adjoint and $T$ is $J$-symmetric. Also, for $f = f_++f_-$ with $f_\pm\in\dom S_\pm$ we have
\begin{align*}
\|JTf\|^2
&= \|Mf_-\|^2 + \|M^*f_+\|^2\le a_-\|f_-\|^2 + b_-\|S_-f_-\|^2 + a_+\|f_+\|^2 + b_+\|S_+f_+\|^2\\
&\le\max\{a_+,a_-\}\|f\|^2 + \max\{b_+,b_-\}\|JS_0f\|^2.
\end{align*}
Hence, $JT$ is $JS_0$-bounded with $JS_0$-bound less than one. By the Kato-Rellich theorem, $J(S_0 + T)$ is self-adjoint, which means that $S = S_0 + T$ is $J$-self-adjoint.

For the proof of \eqref{e:incl} let $\la\in\C\setminus\R$ such that $\|M(S_--\la)^{-1}\|<1$ (i.e., $\|L_-(\la)\|<1$). We have to show that $\la\in\rho(S)$. For this, we make use of the first Schur complement of $S$ which, for $\mu\in\vrho(S_-)$, is defined by
\[
S_1(\mu) = S_+ - \mu + M(S_- - \mu)^{-1}M^*,\qquad\dom S_1(\mu) = \dom S_+.
\]
We shall exploit the fact that $\la\in\vrho(S)$ follows if $S_1(\la)$ is boundedly invertible (see, e.g., \cite[Thm.\ 2.3.3]{t}). Since $\|M(S_--\la)^{-1}\| < 1$, the operator $M(S_--\la)^{-1}M^*$ is $S_+$-bounded with $S_+$-bound $b_+<1$. Hence, $S_1(\la)$ is closed (see \cite[Thm.\ IV.1.1]{k}). Now, for $\mu\in\vrho(S_-)$ we have (cf.\ \eqref{e:adjoint})
\begin{align*}
S_1(\la) - S_1(\mu)
&= \mu - \la + M[(S_- - \la)^{-1} - (S_- - \mu)^{-1}]M^*\\
&= \mu - \la + (\la-\mu)M(S_- - \la)^{-1}(S_- - \mu)^{-1}M^*\\
&= (\la-\mu)\big(L_-(\la) L_-(\ol\mu)^* - I\big)|_{\dom S_+}.
\end{align*}
From this it easily follows that $S_1(\mu)$ is closed for every $\mu\in\vrho(S_-)$. We also have $S_1(\ol\mu)\subset S_1(\mu)^*$. Therefore, for $f_+\in\dom S_+$, $\|f_+\|=1$,
\[
\Im\,(S_1(\la)f_+,f_+) = \frac 1{2i}\big((S_1(\la)-S_1(\ol\la))f_+,f_+\big) = (\Im\la)\big(\big(L_-(\la)L_-(\la)^* - I\big)f_+,f_+\big).
\]
But $\|L_-(\la)\| < 1$. Hence, if $\Im\la > 0$ (for $\Im\la < 0$ a similar argument applies), the numerical range $W(S_1(\la))$ of $S_1(\la)$ is contained in the half-plane $H := \{z : \Im z\le -\delta\}$, where $\delta := (1-\|L_-(\la)\|^2)(\Im\la) > 0$. Hence, for every $\mu\notin H$ and $f_+\in\dom S_+$, $\|f_+\|=1$,
\begin{equation}\label{e:notnew}
\|(S_1(\la)-\mu)f_+\|\,\ge\,|(S_1(\la)f_+,f_+) - \mu\,|\,\ge\,\dist(\mu,H) = \Im\mu + \delta.
\end{equation}
This implies that for each $\mu$ with $\Im\mu > -\delta$ the operator $S_1(\la)-\mu$ is injective and has closed range. Let us now prove that there exists $r > 0$ such that $ri\in\vrho(S_1(\la))$. Then it follows (cf.\ \cite[Ch.\ IV.5.6]{k}) that $0\in\vrho(S_1(\la))$ and thus $\la\in\vrho(S)$. Set $\mu := \la + ri$. Then
\[
S_1(\la)-ri = S_+ - \mu + L_-(\la) M^* = \big(I + L_-(\la) M^*(S_+-\mu)^{-1}\big)(S_+-\mu).
\]
By Lemma \ref{l:norm} the bounded operator $M^*(S_+-\mu)^{-1}$ has norm less than one for sufficiently large $r$. This shows that $S_1(\la)-ri$ is boundedly invertible for such $r$.

Above we have concluded from $\|M(S_--\la)^{-1}\| < 1$ that $\la\in\vrho(S)$. If $\|M^*(S_+-\la)^{-1}\|<1$, one obtains $\la\in\vrho(S)$ by means of analogous arguments, using the second Schur complement $S_2(\la) = S_- - \la + M^*(S_+-\la)^{-1}M$.

Let us now prove the estimate \eqref{e:resolvent} for the resolvent of $S$. For this, let $\la\in\C\setminus\R$ such that $\|L_-(\la)\|<1$ as above. By \cite[Thm.\ 2.3.3]{t} we have
\begin{equation}\label{e:resi}
(S - \la)^{-1}
= \mat{I}{0}{L_-(\ol\la)^*}{I}  \mat{S_1(\la)^{-1}}{0}{0}{(S_--\la)^{-1}}  \mat{I}{-L_-(\la)}{0}{I}.
\end{equation}
Denote the last factor by $L$. Then
\begin{align*}
\|L^*L\|
&= \left\|\mat{I}{0}{-L_-(\la)^*}{I}\mat{I}{-L_-(\la)}{0}{I}\right\| = \left\|\mat{I}{-L_-(\la)}{-L_-(\la)^*}{I + L_-(\la)^*L_-(\la)}\right\|\\
&\le\,\left\|\mat{I}{0}{0}{I}\right\| + \left\|\mat{0}{0}{0}{L_-(\la)^*L_-(\la)}\right\| + \left\|\mat{0}{L_-(\la)}{L_-(\la)^*}{0}\right\| = 1 + \|L_-(\la)\|^2 + \|L_-(\la)\|.
\end{align*}
Note that, since $(S_--\la)^{-1}$ is normal, we have
\[
\|L_-(\ol\la)\| = \|M(S_--\ol\la)^{-1}\| = \|(S_--\la)^{-1}M^*\| = \|(S_--\ol\la)^{-1}M^*\| = \|M(S_--\la)^{-1}\| = \|L_-(\la)\|.
\]
Therefore, for the first factor in \eqref{e:resi} we have the same estimate as for the last. For the middle factor in \eqref{e:resi} we obtain from \eqref{e:notnew} with $\mu=0$ that
\[
\|S_1(\la)^{-1}\|\le\delta^{-1} = |\Im\la|^{-1}(1-\|L_-(\la)\|^2)^{-1}
\]
and therefore \eqref{e:resolvent} in the case $\|L_-(\la)\|<1$ follows. The estimate for the other case can be derived similarly by using the second Schur complement.

In the following last step of this proof we shall show that (i) holds true. The proof of (ii) follows analogous lines. Let $\la\in\R\cap\sigma(S)\cap\vrho(S_-)$ such that $\tau := \|M(S_--\la)^{-1}\| < 1$. We have to show that $\la\in\sp(S)$. For this, assume that $(f_n)\subset\dom S$ is a sequence with $\|f_n\|=1$ and $(S-\la)f_n\to 0$ as $n\to\infty$. Let $f_n = f_n^+ + f_n^-$ with $f_n^\pm\in\dom S_\pm$. Then
\[
-M^*f_n^+ + (S_--\la)f_n^-\to 0,\qquad n\to\infty,
\]
in $\calH_-$. Applying $-(S_--\la)^{-1}$ gives (cf.\ \eqref{e:adjoint})
\[
[M(S_--\la)^{-1}]^*f_n^+ - f_n^-\,\to\,0,\quad n\to\infty.
\]
Set $g_n^- := [M(S_--\la)^{-1}]^*f_n^+$, $n\in\N$. Since both $(g_n^-)$ and $(f_n^-)$ are bounded sequences, we conclude that $\veps_n := \|g_n^-\|^2 - \|f_n^-\|^2\to 0$ as $n\to\infty$ and so
\[
\|f_n^+\|^2 = 1 - \|f_n^-\|^2 = 1 - \|g_n^-\|^2 + \veps_n\,\ge\,1 - \tau^2\|f_n^+\|^2 + \veps_n,
\]
that is, $\|f_n^+\|^2\ge\tfrac{1}{2(1+\tau^2)}$ for $n$ sufficiently large. Therefore, we obtain
\[
[f_n,f_n] = \|f_n^+\|^2 - \|f_n^-\|^2 = \|f_n^+\|^2 - \|g_n^-\|^2 + \veps_n\,\ge\,(1-\tau^2)\|f_n^+\|^2 + \veps_n\,\ge\,\frac{1-\tau^2}{2(1+\tau^2)} + \veps_n
\]
and thus, indeed, $\la\in\sp(S)$.
\end{proof}

Although \eqref{e:incl} in Theorem \ref{t:new} is a fairly accurate estimate on the non-real spectrum, it requires complete knowledge about the norms of $M(S_--\la)^{-1}$ and $M^*(S_+-\la)^{-1}$ for all $\la\in\C\setminus\R$. However, by making use of Lemma \ref{l:norm} we immediately obtain the following theorem, where the spectral inclusion is expressed in terms of the spectra of $S_-$ and $S_+$ instead of parameter-dependent norms.

\begin{thm}\label{t:matrix_gen}
Consider the operator matrix $S$ in \eqref{e:S} with self-adjoint diagonal entries $S_+$ and $S_-$ and assume that
\[
\|Mf_-\|^2\,\le\,a_-\|f_-\|^2 + b_-\|S_-f_-\|^2,\quad f_-\in\dom S_-\subset\dom M,
\]
and
\[
\|M^*f_+\|^2\,\le\,a_+\|f_+\|^2 + b_+\|S_+f_+\|^2,\quad f_+\in\dom S_+\subset\dom M^*,
\]
where $a_\pm,b_\pm\ge 0$, $b_\pm < 1$. Then
\begin{equation}\label{e:hyp_int}
\sigma(S)\setminus\R\,\subset\,\left(\bigcup_{t\in\sigma(S_-)}B_{\sqrt{a_-+b_-t^2}}(t)\right)\,\cap\,\left(\bigcup_{t\in\sigma(S_+)}B_{\sqrt{a_++b_+t^2}}(t)\right).
\end{equation}
Moreover, setting $\Delta_\pm(t) := [t-\sqrt{a_\pm+b_\pm t^2},t+\sqrt{a_\pm+b_\pm t^2}]$, we have
\begin{enumerate}
\item[\rm (i)]  $\R\setminus\bigcup_{t\in\sigma(S_-)}\Delta_-(t)$ is of positive type with respect to $S$.
\item[\rm (ii)] $\R\setminus\bigcup_{t\in\sigma(S_+)}\Delta_+(t)$ is of negative type with respect to $S$.
\end{enumerate}
\end{thm}

\begin{rem}
If $M$ is bounded (i.e., $b_\pm=0$ and $a_\pm = \|M\|^2$), Theorem \ref{t:matrix_gen} implies that $\sigma(S)\setminus\R\subset B_{\|M\|}(\sigma(S_-))\cap B_{\|M\|}(\sigma(S_+))$, that $\R\setminus B_{\|M\|}(\sigma(S_-))$ is of positive type and $\R\setminus B_{\|M\|}(\sigma(S_+))$ is of negative type with respect to $S$. This result was already obtained in \cite[Thm.\ 3.5]{bpt} and more general spectral enclosures for $\sigma(S)\setminus\R$ have recently been found in \cite{glmpt}. Analogous methods might also lead to more general enclosures in the relatively bounded case. However, to avoid technical details we shall not touch this topic here.
\end{rem}

For a short discussion of \eqref{e:hyp_int}, assume for simplicity that $a := a_+=a_-$ and $b := b_+=b_-$. Clearly, if, e.g., $\sigma(S_-) = \sigma(S_+) = \R$, the intersection in \eqref{e:hyp_int} does not improve the unions. On the other hand, if, e.g., $S_-$ is bounded from above by $\gamma_-$ and $S_+$ is bounded from below by $\gamma_+$, we obtain from \eqref{e:hyp_int} that
\[
\sigma(S)\setminus\R\,\subset\,\left(\bigcup_{t\le\gamma_-}B_{\sqrt{a+bt^2}}(t)\right)\,\cap\,\left(\bigcup_{t\ge\gamma_+}B_{\sqrt{a+bt^2}}(t)\right).
\]
That is,
\[
\sigma(S)\setminus\R\,\subset\,\bigcup_{t\in [\gamma_+,\gamma_-]}B_{\sqrt{a+bt^2}}(t)
\]
if $\gamma_+\le\gamma_-$ and
\[
\sigma(S)\setminus\R\,\subset\,B_{\sqrt{a+b\gamma_+^2}}(\gamma_+)\,\cap\,B_{\sqrt{a+b\gamma_-^2}}(\gamma_-)
\]
if $\gamma_+ > \gamma_-$. The following corollary treats the case where, in addition, $\gamma_- = \gamma$ and $\gamma_+ = -\gamma$ for some $\gamma\ge 0$, which becomes relevant in the next section.

\begin{cor}\label{c:matrix}
Consider the operator matrix $S$ in \eqref{e:S} with $S_+\ge-\gamma$ and $S_-\le\gamma$ with some $\gamma\ge 0$ and assume that
\[
\|Mf_-\|^2\,\le\,a\|f_-\|^2 + b\|S_-f_-\|^2,\quad f_-\in\dom S_-\subset\dom M,
\]
and
\[
\|M^*f_+\|^2\,\le\,a\|f_+\|^2 + b\|S_+f_+\|^2,\quad f_+\in\dom S_+\subset\dom M^*,
\]
where $a,b\ge 0$, $b < 1$. Then the operator $S$ in \eqref{e:S} is $J$-non-negative over $\ol\C\setminus K$, where
\begin{equation}\label{e:knochen}
K = \bigcup_{t\in [-\gamma,\gamma]} B_{\sqrt{a+bt^2}}(t),
\end{equation}
and is regular at $\infty$.
\end{cor}

\begin{figure}[ht]
\begin{center}
\includegraphics[scale=.25]{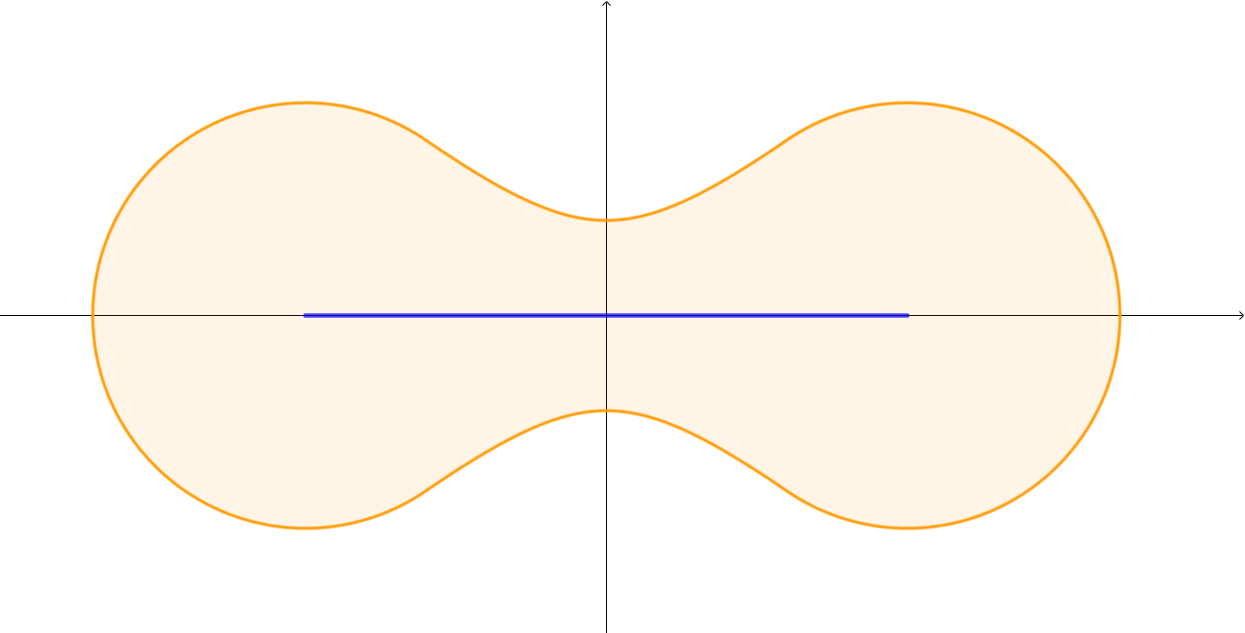}
\end{center}
\caption{The set $K$ in \eqref{e:knochen} with $a = \gamma = 10$ and $b = 0.4$. The blue segment represents the interval $[-\gamma,\gamma]$.}
\end{figure}

\begin{proof}[Proof of Corollary \rmref{c:matrix}]
It follows directly from Theorem \ref{t:matrix_gen} and the preceding discussion that the non-real spectrum of $S$ is contained in $K$. Theorem \ref{t:matrix_gen} also implies that $\R^+\setminus K = (\gamma+\sqrt{a+b\gamma^2},\infty)$ is of positive type and $\R^-\setminus K = (-\infty,-\gamma-\sqrt{a+b\gamma^2})$ is of negative type with respect to $S$.

Let us prove that the growth of the resolvent of $S$ at $\infty$ is of order at most $2$ (cf.\ Definition \ref{d:nnoK} (iii)). In fact, we prove that the order is $1$. For this, let $\la\in\C\setminus\R$ such that $|\la| > \gamma + \sqrt{\gamma^2 + \tfrac ab}$. If $\Re\la\ge 0$, then by Corollary \ref{c:smallerb} we have $\|M(S_--\la)^{-1}\|^2\le b$. Thus, from \eqref{e:resolvent} in Theorem \ref{t:new} we obtain
\[
\|(S-\la)^{-1}\|\,\le\,\frac{3}{|\Im\la|\cdot(1-\|M(S_--\la)^{-1}\|^2)}\,\le\,\frac{3}{1-b}\cdot\frac 1{|\Im\la|}.
\]
A similar reasoning applies to the case where $\Re\la\le 0$.

The regularity of $S$ at $\infty$ follows from \cite[Prop.\ 2.3]{bpt} (see also \cite{cu}) since $J$ is uniformly $J$-positive and leaves $\dom S$ invariant.
\end{proof}

\section{A perturbation result for {\it J}-non-negative operators}
Let $J$ be a self-adjoint involution in the Hilbert space $(\calH,\hproduct)$ inducing the inner product $\product = (J\cdot,\cdot)$ and let $A_0$ be a $J$-non-negative operator in $\calH$ with spectral function $E$. Assume that both $0$ and $\infty$ are not singular critical points of $A_0$ (i.e., $A_0$ is regular) and $0\notin\sigma_p(A_0)$. Then both spectral projections $E_\pm := E(\R^\pm)$ exist, $E_+ + E_- = I$, and $J_0 := E_+ - E_-$ is a bounded and boundedly invertible $J$-non-negative operator which satisfies $J_0^2 = I$. In particular, $J_0$ is both self-adjoint and unitary with respect to the (positive definite) scalar product $\hproduct_0$, where
\[
(f,g)_0 := [J_0f,g],\qquad f,g\in\calH.
\]
By $\|\cdot\|_0$ we denote the norm corresponding to $\hproduct_0$. Since $\|\cdot\|_0$ is equivalent to the original norm, $(\calH,\hproduct_0)$ is a Hilbert space. According to \cite[Lemma 3.8]{bpt} we have
\begin{equation}\label{e:J0}
J_0 = \frac 1\pi\cdot\slim_{n\to\infty}\int_{1/n}^n\Big((A_0 + it)^{-1} + (A_0-it)^{-1}\Big)\,dt,
\end{equation}
where $\slim$ stands for the strong limit. We set
\begin{equation}\label{e:tau_0}
\tau_0 := \|J_0\|.
\end{equation}
Note that from $J_0^2 = I$ it follows that $\tau_0\ge 1$. In what follows, by $T^\ostar$ we denote the adjoint of $T\in L(\calH)$ with respect to the scalar product $\hproduct_0$.

\begin{lem}\label{l:renorm}
Let $T\in L(\calH)$. Then
\[
\|T\|_0\le\tau_0\|T\|.
\]
Moreover for $f\in\calH$, we have
\[
\tau_0^{-1}\|f\|^2\,\le\,\|f\|_0^2\,\le\,\tau_0\|f\|^2
\]
as well as
\[
\|E_\pm f\|_0^2\,\le\,\frac{1+\tau_0}2\,\|f\|^2
\qquad\text{and}\qquad
\|E_\pm f\|^2\,\le\,\frac{1+\tau_0}2\,\|f\|_0^2.
\]
\end{lem}
\begin{proof}
Since the spectrum of a bounded operator is contained in the closure of its numerical range, we have
\[
\|T\|_0^2 = \|T^\ostar T\|_0 = \sup\{|\la| : \la\in\sigma(T^\ostar T)\}\,\le\,\sup\{|(T^\ostar Tf,f)| : f\in\calH,\,\|f\|=1\}\,\le\,\|T^\ostar\|\|T\|.
\]
Hence, from $T^\ostar = J_0JT^*JJ_0$ we obtain $\|T^\ostar\|\le\|J_0\|^2\|T\| = \tau_0^2\|T\|$.

Let now $f\in\calH$. Then we immediately obtain $\|f\|_0^2 = [J_0f,f]\le\|J_0\|\|f\|^2 = \tau_0\|f\|^2$ and, by the first claim,
\[
\|f\|^2 = [f,Jf] = (J_0f,Jf)_0\le\|J\|_0\|f\|_0^2\,\le\,\tau_0\|f\|_0^2.
\]
Moreover,
\[
\|E_{\pm}f\|_0^2 = (E_\pm f,f)_0 = \frac 12(f\pm J_0f,f)_0 = \frac 12[J_0f\pm f,f] = \frac 12(J_0f\pm f,Jf).
\]
Hence, $\|E_{\pm}f\|_0^2\le\tfrac 12\|I\pm J_0\|\|f\|^2\le\frac{1+\tau_0}{2}\|f\|^2$.

Let $P_\pm$ denote the orthogonal projection onto $\calH_\pm := E_\pm\calH$ with respect to the scalar product $\hproduct$. By \cite[Lemma 3.9]{bpt} we have $P_\pm J|_{\calH_\pm} = (E_\pm J|_{\calH_\pm})^{-1}$. Note that $P_+ J|_{\calH_+}$ is a non-negative self-adjoint operator in $(\calH_+,\hproduct)$. Hence, we have
\[
\min\sigma(P_+ J|_{\calH_+}) = \big[\max\sigma(E_+ J|_{\calH_+})\big]^{-1} = \|E_+ J|_{\calH_+}\|^{-1}\,\ge\,\|E_+\|^{-1} = \|\tfrac 12(I+ J_0)\|^{-1}\,\ge\,\frac 2{1+\tau_0}.
\]
Therefore, for $f_+\in\calH_+$,
\[
\|f_+\|_0^2 = [f_+,f_+] = (Jf_+,f_+) = (P_+ Jf_+,f_+)\,\ge\,\frac 2{1+\tau_0}\|f_+\|^2,
\]
and consequently, $\|E_+f\|^2\le\tfrac{1+\tau_0}{2}\|E_+f\|_0^2\le\tfrac{1+\tau_0}{2}\|f\|_0^2$. A similar reasoning applies to $\|E_-f\|^2$.
\end{proof}

We can now state and prove the main theorem in this section.

\begin{thm}\label{t:main}
Let $A_0$ be a regular $J$-non-negative operator in $\calH$ with $0\notin\sigma_p(A_0)$ and let $\tau\ge\tau_0$, where $\tau_0$ is as in \eqref{e:tau_0}. Furthermore, let $V$ be a $J$-symmetric operator in $\calH$ with $\dom A_0\subset\dom V$ such that
\[
(1+\tau)\tau\|Vf\|^2\,\le\,2a\|f\|^2 + b\|A_0f\|^2,\quad f\in\dom A_0,
\]
where $a,b\ge 0$, $b < 1$. Then the operator $A := A_0+V$ is $J$-self-adjoint.

Let $\nu := \inf\{(JVf,f) : f\in\dom V,\,\|f\|=1\}\in [-\infty,\infty)$. If $v\ge 0$, then $A$ is a $J$-non-negative operator. Otherwise, the operator $A$ is $J$-non-negative over
\begin{equation}\label{e:worse}
\ol\C\setminus\bigcup_{t\in[-\gamma,\gamma]}B_{\sqrt{a + bt^2}}(t),
\end{equation}
where\footnote{If $\nu = -\infty$, we set $\gamma= \sqrt{\tfrac{1+\tau}{2\tau}a}$.} $\gamma := \min\{\sqrt{\tfrac{1+\tau}{2\tau}a},-\tfrac{1+\tau}{2}v\}$. In both cases the operator $A$ is regular at $\infty$. If $b < \tfrac{\tau-1}{2\tau}$, then $A$ is $J$-non-negative over
\begin{equation}\label{e:better}
\ol\C\setminus\bigcup_{t\in[-\gamma,\gamma]}B_{\sqrt{\frac{1+\tau}{2\tau(1-b)}(a+bt^2)}}(t).
\end{equation}
\end{thm}

\begin{rem}
If $b < \tfrac{\tau-1}{2\tau}$, then the bound in \eqref{e:better} is better than that in \eqref{e:worse}. If $V$ is bounded (i.e., $b=0$ and $a = \tfrac{(1+\tau)\tau}{2}\|V\|^2$), we obtain from the second part of Theorem \ref{t:main} that $A_0+V$ is $J$-non-negative over
\[
\ol\C\setminus B_{r}([-d,d]),
\]
where $r = \tfrac{1+\tau}{2}\|V\|$ and $d = -\tfrac{1+\tau}{2}\min\sigma(JV)$, which is the same result as \cite[Thm.\ 3.1]{bpt}.
\end{rem}

The following example (cf.\ \cite[Example 3.2]{bpt}) shows that the assumption on the regularity of $A_0$  in Theorem \ref{t:main} cannot be dropped.

\begin{ex}\label{ex:ex}
Let $(\calK,(\cdot,\cdot))$ be a Hilbert space and let $H$ be an unbounded selfadjoint operator in $\calK$ such that $\sigma(H)\subset 
(0,\infty)$. 
Consider the following operators in $\calH := \calK\oplus\calK$:
\begin{equation*}
J=\begin{pmatrix} 0 & I \\ I & 0\end{pmatrix},\quad
A_0=\begin{pmatrix} 0 & I \\ H & 0\end{pmatrix},\quad 
V=\begin{pmatrix} 0 & -I  \\ 0 & 0\end{pmatrix},\quad\text{and}\quad
A_0+V =\begin{pmatrix} 0 & 0 \\ H & 0\end{pmatrix}.
\end{equation*}
It is easy to see that $A_0$ is a $J$-non-negative operator and $V$ is a bounded $J$-selfadjoint operator. Moreover, as $\dom(A_0+V)=\dom A_0 = \dom H\oplus{\calK}$ we conclude $\ran(A_0+V-\lambda)\not=\calH$ for every $\lambda\in\C$, that is, $\sigma(A_0+V)=\C$.
\end{ex}

\begin{proof}[Proof of Theorem \rmref{t:main}]
Without loss of generality we assume that $\dom V = \dom A_0$. As $\tau\ge\tau_0\ge 1$, it follows from the Kato-Rellich theorem that $A = A_0 + V$ is $J$-self-adjoint. Equivalently, $A$ is $J_0$-self-adjoint in $(\calH,\hproduct_0)$.

Let $\oplus_0$ denote the $\hproduct_0$-orthogonal sum of subspaces. Then we have $\calH = \calH_+\oplus_0 \calH_-$ and $\dom A_0 = D_+\oplus_0 D_-$, where $\calH_\pm := (E_\pm\calH,\hproduct_0)$ and $D_\pm = \dom A_0\cap\calH_\pm$. Hence, $V$ is defined on both $D_+$ and $D_-$ and we can write $A_0$ and $V$ as operator matrices
\[
A_0 = \mat{A_+}{0}{0}{A_-}
\qquad\text{and}\qquad
V = \mat{V_+}{M}{N}{V_-},
\]
where
\[
A_\pm = E_\pm(A|_{D_\pm}), \quad V_\pm = E_\pm(V|_{D_\pm}), \quad M = E_+(V|_{D_-}), \quad\text{and}\quad N = E_-(V|_{D_+}).
\]
Note that $A_0$ is $J_0$-non-negative in $(\calH,\hproduct_0)$ which implies that $\pm A_\pm\ge 0$ in $\calH_\pm$. Moreover, since $V$ is $J$-symmetric (and hence $J_0$-symmetric in $(\calH,\hproduct_0)$), we have that $N\subset -M^\ostar$. Hence,
\[
A = \mat{A_++V_+}{M}{-M^\ostar}{A_-+V_-},\qquad\dom A = D_+\oplus_0 D_-.
\]
In particular, $A_\pm + V_\pm$ is a self-adjoint operator in $\calH_\pm$. For $f_\pm\in D_\pm$, we obtain from Lemma \ref{l:renorm} that
\[
\|V_\pm f_\pm\|_0^2 = \|E_\pm Vf_\pm\|_0^2\,\le\,\tfrac{1+\tau_0}{2}\|Vf_\pm\|^2\,\le\,\tfrac a{\tau}\|f_\pm\|^2 + \tfrac b{2\tau}\|A_\pm f_\pm\|^2\,\le\,\tfrac{1+\tau}{2\tau}\big(a\|f_\pm\|_0^2 + \tfrac{b}2\|A_\pm f_\pm\|_0^2\big).
\]
Hence, $V_\pm$ is $A_\pm$-bounded (in $\calH_\pm$) with bounds $\tfrac{1+\tau}{2\tau}a$ and $\tfrac{1+\tau}{4\tau}b$. As $\tfrac{1+\tau}{4\tau}b < 1$, Corollary \ref{c:improvement} implies that the operator $A_-+V_-$ is bounded from above by $\gamma_1 := \sqrt{\tfrac{1+\tau}{2\tau}a}$ and $A_++V_+$ is bounded from below by $-\gamma_1$ (both with respect to $\hproduct_0$). Now, for $f_-\in D_-$ we have
\begin{align*}
\|Mf_-\|_0^2 + \|V_-f_-\|_0^2
&= \|Vf_-\|_0^2\,\le\,\tau_0\|Vf_-\|^2\,\le\,\tfrac{2}{1+\tau}\big(a\|f_-\|^2 + \tfrac{b}2\|A_-f_-\|^2\big)\,\le\,a\|f_-\|_0^2 + \tfrac{b}2\|A_-f_-\|_0^2\\
&\le\,a\|f_-\|_0^2 + b\cdot\big(\|(A_-+V_-)f_-\|_0^2 + \|V_-f_-\|_0^2\big).
\end{align*}
Since $b < 1$, it follows that
\begin{align*}
\|Mf_-\|_0^2\,\le\,a\|f_-\|_0^2 + b\|(A_-+V_-)f_-\|_0^2.
\end{align*}
Similarly, for $f_+\in D_+$ one obtains
\[
\|M^\ostar f_+\|_0^2\,\le\,a\|f_+\|_0^2 + b\|(A_++V_+)f_+\|_0^2.
\]
As seen above, $A_++V_+$ is bounded from below by $-\gamma_1$ and $A_-+V_-$ is bounded from above by $\gamma_1$. Corollary \ref{c:matrix} and Remark \ref{r:J-indep} thus imply that the operator $A$ is $J$-non-negative over $\ol\C\setminus K_1$, where
\[
K_1 = \bigcup_{t\in[-\gamma_1,\gamma_1]}B_{\sqrt{a + bt^2}}(t).
\]
In particular, $\vrho(A)\neq\emptyset$. Hence, if the lower bound $v$ of $JV$ is non-negative, the operator $A$ is $J$-non-negative. Assume that $v < 0$. Then for $f_+\in D_+$ we have
\begin{align*}
((A_++V_+)f_+,f_+)_0
&\ge\,(V_+f_+,f_+)_0 = (Vf_+,E_+f_+)_0 = (Vf_+,J_0f_+)_0\\
&= [Vf_+,f_+] = (JVf_+,f_+)\ge v\|f_+\|^2\,\ge\,\tfrac {1+\tau}2\, v\|f_+\|_0^2,
\end{align*}
and similarly $((A_-+V_-)f_-,f_-)_0\le -\tfrac{1+\tau}{2}v\|f_-\|_0^2$ for $f_-\in D_-$. Thus, $A$ is also $J$-non-negative over $\ol\C\setminus K_2$, where
\[
K_2 = \bigcup_{t\in[-\gamma_2,\gamma_2]}B_{\sqrt{a + bt^2}}(t)
\]
and $\gamma_2 = -\tfrac{1+\tau}{2}v$. This finishes the proof of the first part of the theorem. Since $A_0$ is regular at $\infty$, there exists a uniformly $J$-positive operator $W$ in $\calH$ such that $W\dom A_0\subset\dom A_0$ (see \cite{cu}). From $\dom A = \dom A_0$ it follows that $A$ is regular at $\infty$.

For the second part of the theorem, let $f_-\in D_-$. Then we have
\[
\|V_-f_-\|^2 = \|E_-Vf_-\|^2\,\le\,\tfrac{1+\tau}2\|Vf_-\|_0^2\,\le\,\tfrac{(1+\tau)\tau}2\|Vf_-\|^2.
\]
Therefore,
\begin{align*}
\tfrac{1+\tau}{2}\|Vf_-\|^2
&\le\,\tfrac{a}{\tau}\|f_-\|^2 + \tfrac b{2\tau}\|A_-f_-\|^2\,\le\,\tfrac{a}{\tau}\|f_-\|^2 + \tfrac b{\tau}\|(A_-+V_-)f_-\|^2 + \tfrac b{\tau}\|V_-f_-\|^2\\
&\le\left(\tfrac{(1+\tau)a}{2\tau}\|f_-\|_0^2 + \tfrac{(1+\tau)b}{2\tau}\|(A_-+V_-)f_-\|_0^2\right) + \tfrac{1+\tau}2b\|Vf_-\|^2,
\end{align*}
which implies
\[
\|Mf_-\|_0^2 = \|E_+Vf_-\|_0^2\le\tfrac{1+\tau}2\|Vf_-\|^2\,\le\,\tfrac{(1+\tau)a}{2\tau(1-b)}\|f_-\|_0^2 + \tfrac{(1+\tau)b}{2\tau(1-b)}\|(A_-+V_-)f_-\|_0^2.
\]
Similarly, for $f_+\in D_+$ we have
\[
\|M^\ostar f_+\|_0^2\,\le\,\tfrac{(1+\tau)a}{2\tau(1-b)}\|f_+\|_0^2 + \tfrac{(1+\tau)b}{2\tau(1-b)}\|(A_++V_+)f_+\|_0^2.
\]
We apply Corollary \ref{c:matrix} and conclude that the claim holds if only $\tfrac{(1+\tau)b}{2\tau(1-b)} < 1$, which is equivalent to $b < \tfrac{2\tau}{1+3\tau}$. But since the first part of the theorem yields better bounds for $\tfrac{\tau-1}{2\tau}\le b < \tfrac{2\tau}{1+3\tau}$ (e.g., $b=\tfrac 12$), we require that $b < \tfrac{\tau-1}{2\tau}$ 
\end{proof}

Since the $A_0$-bound of an $A_0$-compact operator is always zero (see \cite[Corollary III.7.7]{ee}), the next corollary immediately follows from Theorem \ref{t:main}.

\begin{cor}
Let $A_0$ be a regular $J$-non-negative operator in $\calH$ with $0\notin\sigma_p(A_0)$ and let $V$ be a $J$-symmetric operator in $\calH$ which is $A_0$-compact. Then the operator $A_0+V$ is $J$-self-adjoint, $J$-non-negative over $\ol\C\setminus K$ for some $K$, and regular at $\infty$.
\end{cor}

\section{An application to singular indefinite Sturm-Liouville operators}\label{s:appl}
In this section we apply the abstract Theorem \ref{t:main} to indefinite Sturm-Liouville operators arising from differential expressions of the form
\[
\calL(f)(x) = \sgn(x)\big(-f''(x) + q(x)f(x)\big),\quad x\in\R\,,
\]
with a real-valued potential $q\in L^p(\R)$, $2\le p\le\infty$. 
The maximal domain corresponding to $\calL$ is given by
\[
D_{\max} = \big\{f\in L^2(\R) : f,f'\in\AC_{\rm loc}(\R),\,-f''+qf\in L^2(\R)\big\},
\]
where $\AC_{\rm loc}(\R)$ stands for the space of functions $f : \R\to\C$ which are locally absolutely continuous. The maximal operator $A$ associated with $\calL$ is defined as
\begin{equation}\label{e:SLA}
Af := \calL(f),\qquad\dom A := D_{\max}.
\end{equation}
We also define the corresponding {\em definite} Sturm-Liouville operator $T$ by
\[
Tf := JAf = -f'' + qf,\qquad \dom T := D_{\max},
\]
where $J : L^2(\R)\to L^2(\R)$ is the operator of multiplication with the function $\sgn(x)$. Note that $J = J^{-1} = J^*$ and hence $J$ is a fundamental symmetry.

If $q\in L^p(\R)$, $2\le p\le\infty$, it is well known that the operator $T$ is self-adjoint and bounded from below. In fact, this holds for much more general differential expressions (cf.\ \cite[Thm.\ 1.1]{bst}). Consequently, the operator $A = JT$ is $J$-self-adjoint and it was shown in \cite[Thm.\ 4.2]{bp} that $A$ is non-negative over $\ol\C\setminus K$ for some compact $K = K^*$. In particular, the non-real spectrum of $A$ is contained in $K$. It is also known (cf.\ \cite{bp}) that the non-real spectrum of $A$ consists of isolated eigenvalues which can only accummulate to $[\mu,-\mu]$, where $\mu$ is the infimum of the essential spectrum of $T$. However, the set $K$ could not explicitly be specified in \cite{bp}.

Recently, in \cite{bst} spectral enclosures for the non-real spectrum have been found: $\sigma(A)\setminus\R\subset K'$. However, this does not mean that the operator $A$ is non-negative over $\ol\C\setminus K'$ for $K'$ might be much smaller than $K$. In Theorem \ref{t:SL} below, we provide explicit bounds on the compact set $K$, thereby improving the bounds from \cite{bst} if $q\le 0$.

The following lemma sets the ground for the applicability of Theorem \ref{t:main} to the problem. In its proof we make use of the Fourier transform. Here, we use the definition
\[
\calF f(\omega) = \hat f(\omega) = \int_\R f(x)e^{-2\pi ix\omega}\,dx,\qquad \omega\in\R,
\]
for Schwartz functions $f$. The Fourier transform $\calF$ is extended to a unitary operator in $L^2(\R)$ in the usual way. By $H^m(\R)$ we denote the Sobolev space of regularity order $m > 0$ on $\R$, i.e.,
\[
H^m(\R) = \big\{f\in L^2(\R) : (1+\omega^2)^{m/2}\hat f\in L^2(\R)\big\}.
\]
Recall that for $f\in H^1(\R)$ we have $\wh{f'}(\omega) = 2\pi i\omega\hat f(\omega)$ and thus, for $f\in H^2(\R)$, $\wh{f''}(\omega) = -4\pi^2\omega^2\hat f(\omega)$.

\begin{lem}\label{l:Ls}
Let $p\in [2,\infty]$. Then for any $g\in L^p(\R)$ and any $f\in H^2(\R)$ we have $fg\in L^2(\R)$ and for any $r > 0$ the following inequality holds:
\[
\|fg\|_2\,\le\,(2r)^{1/p}\left(\|f\|_2 + \frac{1}{2\sqrt{3}\pi^2pr^2}\,\|f''\|_2\right)\|g\|_p.
\]
\end{lem}
\begin{proof}
We first observe that for $f\in H^2(\R)$ we have $\omega^2\hat f\in L^2(\R)$ and hence, for $r>0$,
\begin{align}
\begin{split}\label{e:H2L1}
\|\hat f\|_1
&= \int_{-r}^r|\hat f(\omega)|\,d\omega + \int_{|\omega|>r}(4\pi^2\omega^2)^{-1}\cdot |4\pi^2\omega^2\hat f(\omega)|\,d\omega\\
&\le \sqrt{2r}\,\|\hat f\|_2 + \left(2\int_r^\infty\frac{d\omega}{16\pi^4\omega^4}\,\right)^{1/2}\left(\int_{|\omega|>r}|4\pi^2\omega^2\hat f(\omega)|^2\,d\omega\right)^{1/2}\\
&\le \sqrt{2r}\,\|f\|_2 + \frac{\sqrt 2}{4\sqrt 3 \pi^2}\,r^{-3/2}\|f''\|_2 = \sqrt{2r}\;\,C_r(f),
\end{split}
\end{align}
where $C_r(f) := \|f\|_2 + \frac{1}{4\sqrt 3\pi^2r^2}\|f''\|_2$.
Thus, for $g\in L^2(\R)$ we obtain
\[
\|gf\|_2\,\le\,\|f\|_\infty\|g\|_2\,\le\,\|\wh f\|_1\|g\|_2,
\]
which yields the desired inequality for $p=2$. For $p=\infty$ the inequality is evident.

Now, let $p\in (2,\infty)$ and fix $f\in H^2(\R)$. Set $T_fg := fg$. Then $T_f$ is bounded both as an operator from $L^\infty(\R)$ to $L^2(\R)$ (with norm $\|f\|_2$) and from $L^2(\R)$ to $L^2(\R)$ (with norm $\le \sqrt{2r}\;C_r(f)$). By the Riesz-Thorin interpolation theorem (see, e.g., \cite[Thm.\ 6.27]{f}), $T_f$ is also bounded as an operator from $L^{2/\theta}(\R)$ to $L^2(\R)$, $\theta\in (0,1)$, with norm
\[
\|T_f\|_{L^{2/\theta}(\R)\to L^2(\R)}\,\le\,\|f\|_2^{1-\theta}\left((2r)^{1/2}C_r(f)\right)^\theta = (2r)^{\theta/2}\|f\|_2\left(1 + \frac{\|f\|_2^{-1}\|f''\|_2}{4\sqrt 3\pi^2r^2}\right)^\theta.
\]
Now, the claim follows from setting $p=2/\theta$ and the simple inequality $(1+c)^\theta\le 1+c\theta$, which holds for $\theta\in(0,1)$ and $c>0$.
\end{proof}

\begin{thm}\label{t:SL}
Let $q\in L^p(\R)$, $p\in [2,\infty)$, and set
\[
s_p := 4 - 3\sqrt{2} - 5 p + 4\sqrt{2}p + \sqrt{
 44 - 31\sqrt{2} - 88 p + 62\sqrt{2}p + 57 p^2 - 40\sqrt{2}p^2} 
\]
\[
f(s) := \sqrt{2\,\frac{(17+12\sqrt{2})s + 4 + 3\sqrt{2}}{(3 + 2\sqrt{2})s - 1 - \sqrt{2}}},
\]
as well as
\[
C_p := (1+\sqrt 2)\sqrt{3-2\sqrt 2}\,\sqrt{\tfrac{2p}{2p-1}}\,\left(\tfrac{16\cdot\sqrt 2(3+2\sqrt 2)^2}{3\pi^4 p}s_p\right)^{\frac{1}{4p-2}}.
\]
Then the singular indefinite Sturm-Liouville operator $A$ from \eqref{e:SLA} with potential $q$ is $J$-non-negative over $\ol\C\setminus K$, where
\begin{equation}\label{e:SL_incl}
K := \left\{\la\in\C : |\Im\la|\le C_pf(s_p)\|q\|_p^{\frac{2p}{2p-1}},\;\;|\Re\la|\le C_p\big(\sqrt{6+4\sqrt 2}+f(s_p)\big)\|q\|_p^{\frac{2p}{2p-1}}\right\}.
\end{equation}
\end{thm}
\begin{proof}
The differential operator $A_0$, defined by
\[
A_0f := -\sgn\cdot f'',\qquad\dom A_0 := H^2(\R),
\]
is $J$-self-adjoint. In fact, $A_0$ is $J$-non-negative and neither $0$ nor $\infty$ is a singular critical point of $A_0$, see \cite{cn}. Also, obviously, $0\notin\sigma_p(A_0)$. 

{\it Step 1. \braces{Calculation of $\tau_0$}} We have $\tau_0 = \|J_0\|$, where (see \eqref{e:J0})
\[
J_0 = \frac 1\pi\cdot\slim_{n\to\infty}\int_{1/n}^n\Big((A_0 + it)^{-1} + (A_0-it)^{-1}\Big)\,dt,
\]
According to \cite[Proof of Thm.\ 4.2]{bpt} for $f\in L^2(\R)$ we have
\begin{equation}\label{e:J0_SL}
[J_0f,f] = \|f\|_2^2 + \frac 1{\pi}\cdot\lim_{n\to\infty}\Re\int_{1/n}^n\sqrt{\tfrac 2t}\cdot [\phi_{t},f][\phi_{t},\ol f]\,dt,
\end{equation}
where $\phi_{t}(x) = e^{(i-\sgn(x))x\cdot\sqrt{t/2}}$. Setting $f_1(x) = f(x)$ and $f_2(x) := f(-x)$, we obtain
\[
[\phi_t,f] = \int_0^\infty\big(e^{-\sqrt{\frac t2}(1-i)x}\ol{f_1(x)} - e^{-\sqrt{\frac t2}(1+i)x}\ol{f_2(x)}\big)\,dx,
\]
and thus
\[
[\phi_t,f][\phi_t,\ol f] = \sum_{j,k=1}^2(-1)^{j+k}\int_0^\infty\!\!\!\!\!\int_0^\infty e^{-z_{j,k}(x,y)\sqrt{\frac t2}}\cdot\ol{f_j(x)}f_k(y)\,dx\,dy,
\]
where $z_{jk}(x,y) = x+y + i((-1)^jx + (-1)^ky)$. Taking into account that
\[
\lim_{n\to\infty}\int_{1/n}^n\sqrt{\tfrac 2t}\cdot e^{-z\sqrt{t/2}}\,dt = \tfrac 4z
\]
for any $z\in\C$ with $\Re z > 0$, an application of Fubini's theorem and the dominated convergence theorem yields
\[
\lim_{n\to\infty}\int_{1/n}^n\sqrt{\tfrac 2t}\cdot [\phi_{t},f][\phi_{t},\ol f]\,dt = 4\sum_{j,k=1}^2(-1)^{j+k}\int_0^\infty\!\!\!\!\!\int_0^\infty\frac{\ol{f_j(x)}f_k(y)}{z_{j,k}(x,y)}\,dx\,dy.
\]
Taking real parts, we find that (see \eqref{e:J0_SL})
\[
[J_0f,f] = \|f\|_2^2 + \frac 2\pi\int\limits_0^\infty\!\!\!\int\limits_0^\infty\left[\frac{\ol{f_1(x)}f_1(y) + \ol{f_2(x)}f_2(y)}{x+y} - \frac{2(x+y)}{x^2+y^2}\Re(\ol{f_1(x)}f_2(y))\right]\,dx\,dy.
\]
For $g,h\in L^2(\R^+)$ and a measurable symmetric kernel $L : \R^+\times\R^+\to [0,\infty)$ satisfying the homogeneity condition $L(tx,ty) = t^{-1}L(x,y)$ formally define the sesquilinear form
\[
\<g,h\>_L := \int_0^\infty\int_0^\infty L(x,y)g(x)\ol{h(y)}\,dx\,dy.
\]
By \cite[Thm.\ 319]{hlp}, the form $\aproduct_L : L^2(\R^+)\times L^2(\R^+)\to\C$ is well-defined and bounded if $\|L\|_* := \int_0^\infty L(x,1)x^{-1/2}\,dx < \infty$. In this case, its bound is given by $\|L\|_*$. Here, we shall consider the kernels
\[
L_1(x,y) = \frac 1{x+y},\quad L_2(x,y) = \frac{2(x+y)}{x^2+y^2},\quad\text{and}\quad L_3 = 2L_1 + L_2.
\]
For these we have the bounds $\|L_1\|_* = \pi$, $\|L_2\|_* = 2\sqrt 2\pi$, and $\|L_3\|_* = 2(1+\sqrt 2)\pi$. Thus,
\begin{align*}
[J_0f,f]
&= \|f\|_2^2 + \tfrac 2\pi(\<f_1,f_1\>_{L_1} + \<f_2,f_2\>_{L_1} - \Re\,\<f_1,f_2\>_{L_2})\\
&\le\|f\|_2^2 + 2\big(\|f_1\|_{L^2(\R^+)}^2 + \|f_2\|_{L^2(\R^+)}^2 + 2\sqrt 2\|f_1\|_{L^2(\R^+)}\|f_2\|_{L^2(\R^+)}\big)\,\le\,(3+2\sqrt 2)\|f\|_2^2,
\end{align*}
and therefore
\[
\tau_0 = \|JJ_0\| = \sup\{(JJ_0f,f)_2 : \|f\|_2=1\} = \sup\{[J_0f,f] : \|f\|_2=1\}\,\le\,3+2\sqrt 2.
\]
On the other hand, if we choose functions $f\in L^2(\R)$ with $f_2 = -f_1$, we obtain
\[
[J_0f,f] = \|f\|_2^2 + \frac 2\pi\int_0^\infty\int_0^\infty\left(\frac 2{x+y} + \frac{2(x+y)}{x^2+y^2}\right)\ol{f_1(x)}f_1(y)\,dx\,dy = \|f\|_2^2 + \tfrac 2\pi\<f_1,f_1\>_{L_3}.
\]
For $\veps>0$, choosing a function $f_1\in L^2(\R^+)$, $f_1\ge 0$, with $\|f_1\|_{L^2(\R^+)}^2 = \tfrac 12$ (i.e., $\|f\|_2=1$) and $\<f_1,f_1\>_{L_3}\ge \tfrac 12(\|L_3\|_*-\pi\veps)$ leads to $[J_0f,f]\ge 3+2\sqrt 2 - \veps$, showing that
\[
\tau_0 = 3 + 2\sqrt 2.
\]

{\it Step 2. \braces{Estimation of the exceptional region $K$}} Let the operator $V$ be defined by
\[
Vf := \sgn\cdot q\cdot f,\qquad\dom V = \{f\in L^2(\R) : qf\in L^2(\R)\}.
\]
By Lemma \ref{l:Ls} we have $\dom A_0\subset\dom V$ and, for any $r > 0$ and $f\in\dom A_0$,
\[
\|Vf\|_2\,\le\,(2r)^{1/p}\,\left(\|f\|_2 + \frac{1}{2\sqrt 3\pi^2pr^2}\|A_0f\|_2\right)\|q\|_p.
\]
This implies that for all $r,\kappa > 0$ and all $f\in\dom A_0$ we have
\begin{equation}\label{e:V}
\|Vf\|_2^2\,\le\,\|q\|_p^24^{\frac 1p}(1+\tfrac 1\kappa)r^{\frac 2p}\|f\|_2^2 + \frac{(1+\kappa)4^{\frac 1p}\|q\|_p^2}{12\,\pi^4p^2}r^{\frac 2p-4}\,\|A_0f\|_2^2.
\end{equation}
In particular, $A_0+V$ is $J$-self-adjoint by the Kato-Rellich theorem. By \cite[Thm.\ 1.1]{bst} the same is true for the operator $A$. And since $\dom(A_0+V) = H^2(\R)\subset D_{\max} = \dom A$, it follows that $A_0+V\subset A$ and thus $A_0+V = A$.

From \eqref{e:V} we get that with $\tau := \tau_0 = 3+2\sqrt 2$ we have
\begin{align*}
(1+\tau)\tau\,\|Vf\|_2^2\,\le\,2\alpha_{r,\kappa}\|f\|_2^2 + \beta_{r,\kappa}\|A_0f\|_2^2,\qquad f\in\dom A_0,
\end{align*}
where
\[
\alpha_{r,\kappa} = \tfrac{(1+\tau)\tau}2\|q\|_p^24^{\frac 1p}(1+\tfrac 1\kappa)r^{\frac 2p}
\qquad\text{and}\qquad
\beta_{r,\kappa} = \frac{(1+\tau)\tau(1+\kappa)4^{\frac 1p}\|q\|_p^2}{12\,\pi^4p^2}\,r^{\frac 2p-4}.
\]
We have $\beta_{r,\kappa} < \tfrac{\tau-1}{2\tau}$ if and only if $r > r_\kappa := (\tfrac{\tau^2(1+\tau)(1+\kappa)4^{1/p}}{6(\tau-1)\pi^4p^2}\|q\|_p^2)^{1/(4-2/p)}$. Therefore, for $s>1$ we set $r_{\kappa,s} := (\tfrac{\tau^2(1+\tau)(1+\kappa)4^{1/p}}{6(\tau-1)\pi^4p^2}\|q\|_p^2s)^{1/(4-2/p)}$ as well as
\[
a_{\kappa,s} := \alpha_{r_{\kappa,s},\kappa} = \tfrac{(1+\tau)\tau}{2}\left(\tfrac{8\tau^2(1+\tau)}{3(\tau-1)\pi^4p^2}\right)^{\frac 1{2p-1}}\|q\|_p^{\frac{4p}{2p-1}}(1+\tfrac 1\kappa)(1+\kappa)^{\frac 1{2p-1}}s^{\frac 1{2p-1}},
\]
\[
b_s := \beta_{r_{\kappa,s},\kappa} = \tfrac{\tau-1}{2\tau\cdot s}
\qquad\text{and}\qquad
\gamma_{\kappa,s} := \sqrt{\tfrac{1+\tau}{2\tau}a_{\kappa,s}}.
\]
The minimum of $(1+\tfrac 1\kappa)(1+\kappa)^{\frac 1{2p-1}}$ is attained at $\kappa = 2p-1$ and equals $\tfrac{2p}{2p-1}(2p)^{1/(2p-1)}$. We choose this $\kappa$ and obtain $a_s = M_ps^{1/(2p-1)}$, where
\[
M_p := \tfrac{(1+\tau)\tau}{2}\left(\tfrac{16\tau^2(1+\tau)}{3(\tau-1)\pi^4p}\right)^{\frac 1{2p-1}}\tfrac{2p}{2p-1}\|q\|_p^{\frac{4p}{2p-1}} = \tfrac{4\tau^2}{1+\tau}\,s_p^{-\frac{1}{2p-1}}\|q\|_p^{\frac{4p}{2p-1}}C_p^2,
\]
as well as $\gamma_s = \sqrt{\tfrac{1+\tau}{2\tau}a_s}$. Now, the second part of Theorem \ref{t:main} implies that for each $s>1$ the operator $A = A_0 + V$ is $J$-non-negative over $\ol\C\setminus K_{s}$, where
\[
K_{s} := \bigcup_{t\in[-\gamma_{s},\gamma_{s}]}B_{\sqrt{\frac{1+\tau}{2\tau(1-b_s)}(a_s+b_st^2)}}(t).
\]
For any $\la\in K_s$ we have
\[
|\Im\la|^2\le\tfrac{1+\tau}{2\tau(1-b_s)}(a_s+b_s\gamma_s^2) = \tfrac{1+\tau}{2\tau-\frac{\tau-1}{s}}\left(1+\tfrac{\tau^2-1}{4\tau^2s}\right)a_s = \tfrac{1+\tau}{4\tau^2}M_pf(s)^2s^{\frac 1{2p-1}}.
\]
The minimum of $s\mapsto f(s)^2s^{1/(2p-1)}$ on $(1,\infty)$ is attained at $s_p$. Choosing $s = s_p$, we find that $A$ is $J$-non-negative over $\ol\C\setminus K'$, where $K' := K_{s_p}$, and we have just proved the claimed bound on $|\Im\la|$ for $\la\in K'$. Furthermore, for $\la\in K'$ we have
\begin{align*}
|\Re\la|
&\le\gamma_{s_p} + \sqrt{\tfrac{(1+\tau)(a_{s_p}+b_{s_p}\gamma_{s_p}^2)}{2\tau(1-b_{s_p})}} = \sqrt{\tfrac{1+\tau}{2\tau}}\sqrt{\tfrac{4\tau^2}{1+\tau}\|q\|_p^{\frac{4p}{2p-1}}C_p^2} + C_pf(s_p)\|q\|_p^{\frac{2p}{2p-1}}\\
&= C_p\big(\sqrt{2\tau}+f(s_p)\big)\|q\|_p^{\frac{2p}{2p-1}},
\end{align*}
and the theorem is proved.
\end{proof}

\begin{rem}
(a) The bound on the real part in \eqref{e:SL_incl} can be further slightly improved by minimizing the expression $\gamma_s + \sqrt{\frac{1+\tau}{2\tau(1-b_s)}(a_s+b_s\gamma_s^2)}$, where $\tau = 3+2\sqrt 2$.

(b) In \cite{bpt} it was proved that $\tau_0\le 9$. In the proof of Theorem \ref{t:SL} we have now shown that $\tau_0 = 3+2\sqrt 2$.
%

(c) Estimates on the non-real spectrum of singular indefinite Sturm-Liouville operators have been obtained in \cite{bst} for various weights and potentials. In the case of the signum function as weight and a negative potential $q\in L^p(\R)$ the enclosure in \cite[Cor.\ 2.7 (ii)]{bst} for the non-real eigenvalues $\la$ of $A$ reads as follows:
\begin{equation}\label{e:bst}
|\Im\la|\le 2^{\frac{2p+1}{2p-1}}\cdot 3\cdot\sqrt{3}\,\|q\|_p^{\frac{2p}{2p-1}}
\quad\text{and}\quad
|\la|\le\left(2^{\frac{2p+1}{2p-1}}\cdot 3\cdot\sqrt{3} + 2^{\frac{3-2p}{2p-1}}\cdot 9\right)\,\|q\|_p^{\frac{2p}{2p-1}}.
\end{equation}
A direct comparison shows that the enclosure for the non-real spectrum of $A$ in Theorem \ref{t:SL} is strictly better in the sense that the region $K$ in \eqref{e:SL_incl} is properly contained in that described by \eqref{e:bst} (see Figure \ref{f:comp_bst}). This is remarkable inasmuch as our bound \eqref{e:SL_incl} was mainly obtained by applying the abstract perturbation result Theorem \ref{t:main}, whereas in \cite{bst} the authors work directly with the differential expressions. It is also noteworthy that the estimates in \cite[Cor.\ 2.7 (ii)]{bst} are of the same form $C\|q\|_p^{2p/(2p-1)}$ as in \eqref{e:SL_incl}.

\begin{figure}[ht]
\begin{center}
\includegraphics[scale=.4]{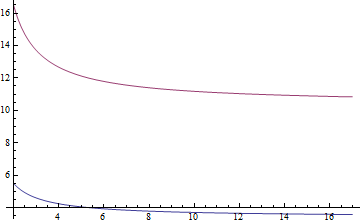}
\hspace{2cm}
\includegraphics[scale=.4]{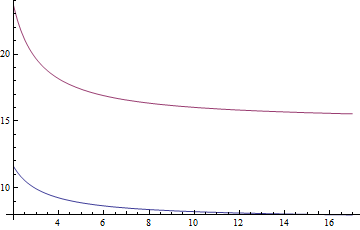}
\end{center}
\caption{Comparison of the enclosures in \eqref{e:SL_incl} (blue) and \eqref{e:bst} (red). The left figure depicts the bounds on the imaginary part in dependence of $p$, respectively, whereas the right figure compares the bound on $|\la|$ in \eqref{e:bst} and half of the diagonal of the rectangle $K$ in \eqref{e:SL_incl}. The limits of the blue curves as $p\to\infty$ are $2+\sqrt 2\approx 3.41$ and $\sqrt{30+20\sqrt 2}\approx 7.63$, resp., whereas the limits of the red curves are given by $6\sqrt 3\approx 10.39$ and $6\sqrt 3 + 4.5\approx 14.89$, resp.}\label{f:comp_bst}
\end{figure}

(d) Recently, the bound on the imaginary part of the eigenvalues of $A$ could be further significantly improved in \cite{ci} by using a Birman-Schwinger type principle. However, the bounding region in \cite{ci} is not compact. To be precise, it was shown that each eigenvalue $\la$ of $A$ in \eqref{e:SLA} satisfies $2^{\frac 3{2p}-1}|\la|^{\frac 1p}|\Im\la|^{1-\frac 1p}\,\le\,\big(|\la|+|\Re\la|\big)^{\frac 1{2p}}\|q\|_p$, which implies
\[
|\Im\la|\,\le\,2\left(\frac{3\sqrt 3}{16}\right)^{\frac 1{2p-1}}\|q\|_p^{\frac{2p}{2p-1}}.
\]
\end{rem}

\section*{Acknowledgments}
The author thanks Jussi Behrndt, Christian G\'erard, David Krej\v{c}i\v{r}\'ik, Ilya Krishtal, Christiane Tretter, and Carsten Trunk for valuable discussions and hints.

\section*{Author Affiliation}
\end{document}